\documentclass[12pt]{amsart}

\usepackage{amssymb}
\usepackage{bm}
\usepackage{graphicx}
\usepackage{psfrag}
\usepackage{color}
\usepackage{url}
\usepackage{algpseudocode}
\usepackage{fancyhdr}
\usepackage{xy}
\input xy
\xyoption{all}

\newtheorem*{maintheorem*}{Main Theorem}
\newtheorem{theorem}{Theorem}[section]
\newtheorem{prop}[theorem]{Proposition}
\newtheorem{lemma}[theorem]{Lemma}
\newtheorem{cor}[theorem]{Corollary}
\theoremstyle{definition}
\newtheorem{definition}[theorem]{Definition}
\newtheorem{example}[theorem]{Example}
\newtheorem{remark}[theorem]{Remark}
\numberwithin{equation}{section}

\newcommand{\rr}{\mathbb{R}}
\newcommand{\zz}{\mathbb{Z}}

\newcommand{\fivepf}[9]
{
	\left\{
	\begin{array}{ll}
		#1 & \mbox{if } #2 \\
		#3 & \mbox{if } #4 \\
		#5 & \mbox{if $i=1$} \\
		#6 & \mbox{if } #7 \\
		#8 & \mbox{if } #9
	\end{array}
	\right.
}


\begin{document}
	
	\mbox{}
	\title{Positroids Induced by Rational Dyck Paths}
	\author{Felix Gotti}
	\address{Department of Mathematics\\UC Berkeley\\Berkeley, CA 94720}
	\email{felixgotti@berkeley.edu}
	\date{\today}
	
	\begin{abstract}
		A rational Dyck path of type $(m,d)$ is an increasing unit-step lattice path from $(0,0)$ to $(m,d) \in \zz^2$ that never goes above the diagonal line $y = (d/m)x$. On the other hand, a positroid of rank $d$ on the ground set $[d+m]$ is a special type of matroid coming from the totally nonnegative Grassmannian. In this paper we describe how to naturally assign a rank $d$ positroid on the ground set $[d+m]$, which we name \emph{rational Dyck positroid}, to each rational Dyck path of type $(m,d)$. We show that such an assignment is one-to-one. There are several families of combinatorial objects in one-to-one correspondence with the set of positroids. Here we characterize some of these families for the positroids we produce, namely Grassmann necklaces, decorated permutations, \reflectbox{L}-diagrams, and move-equivalence classes of plabic graphs. Finally, we describe the matroid polytope of a given rational Dyck positroid.
	\end{abstract}

\maketitle

\tableofcontents

\section{Introduction} \label{sec:intro}

For each pair of nonnegative integers $(m,d)$, a \emph{rational Dyck path} of type $(m,d)$ is a lattice path from $(0,0)$ to $(m,d)$ whose steps are either horizontal or vertical subject to the restriction that it never goes above the line $y = (d/m)x$. Figure~\ref{fig:rational Dyck path} below depicts a rational Dyck path of type $(8,5)$. Note that a rational Dyck path of type $(m,m)$ is just an ordinary Dyck path of length $2m$. The number of Dyck paths of length $2m$ is precisely the $m$-th Catalan number (many other families of relevant combinatorial objects can also be counted by the Catalan numbers; see~\cite{rS15}). The number $\text{Cat}(m,d)$ of rational Dyck paths of type $(m,d)$ is the \emph{rational Catalan number} associated to the pair $(m,d)$. It is known that
\begin{equation} \label{eq:formula of rational catalan number coprime}
	\text{Cat}(m,d) = \frac{1}{d+m} \binom{d+m}{d} 
\end{equation}
when $\gcd(d,m) = 1$. A general formula for the rational Catalan numbers (without assuming co-primeness) was first conjectured by Grossman \cite{hG50} and then proved by Bizley \cite{mB54}. This general formula is much more involved than the one stated in \eqref{eq:formula of rational catalan number coprime}, as the next generating function shows:
\[
	\sum_{n=0}^\infty \text{Cat}(nm,nd) x^n = \exp\bigg( \sum_{j=1}^\infty \frac 1{d+m}\binom{jd + jm}{jd}\frac{x^j}j \bigg),
\]
where, as before, $\gcd(d,m) = 1$. Rational Catalan numbers also appear in the context of partitions. An $(m,d)$-\emph{core} is a partition having no hook lengths equal to $d$ or $m$. A bijection from the set of rational Dyck paths of type $(m,d)$ to the set of $(m,d)$-cores has been established by Anderson in \cite{jA02}. For several results and conjectures on $(m,d)$-cores, the reader can consult \cite{ACHJ14} and \cite{SZ15}.


\begin{figure}[h]
	\centering
	\includegraphics[width = 7.5cm]{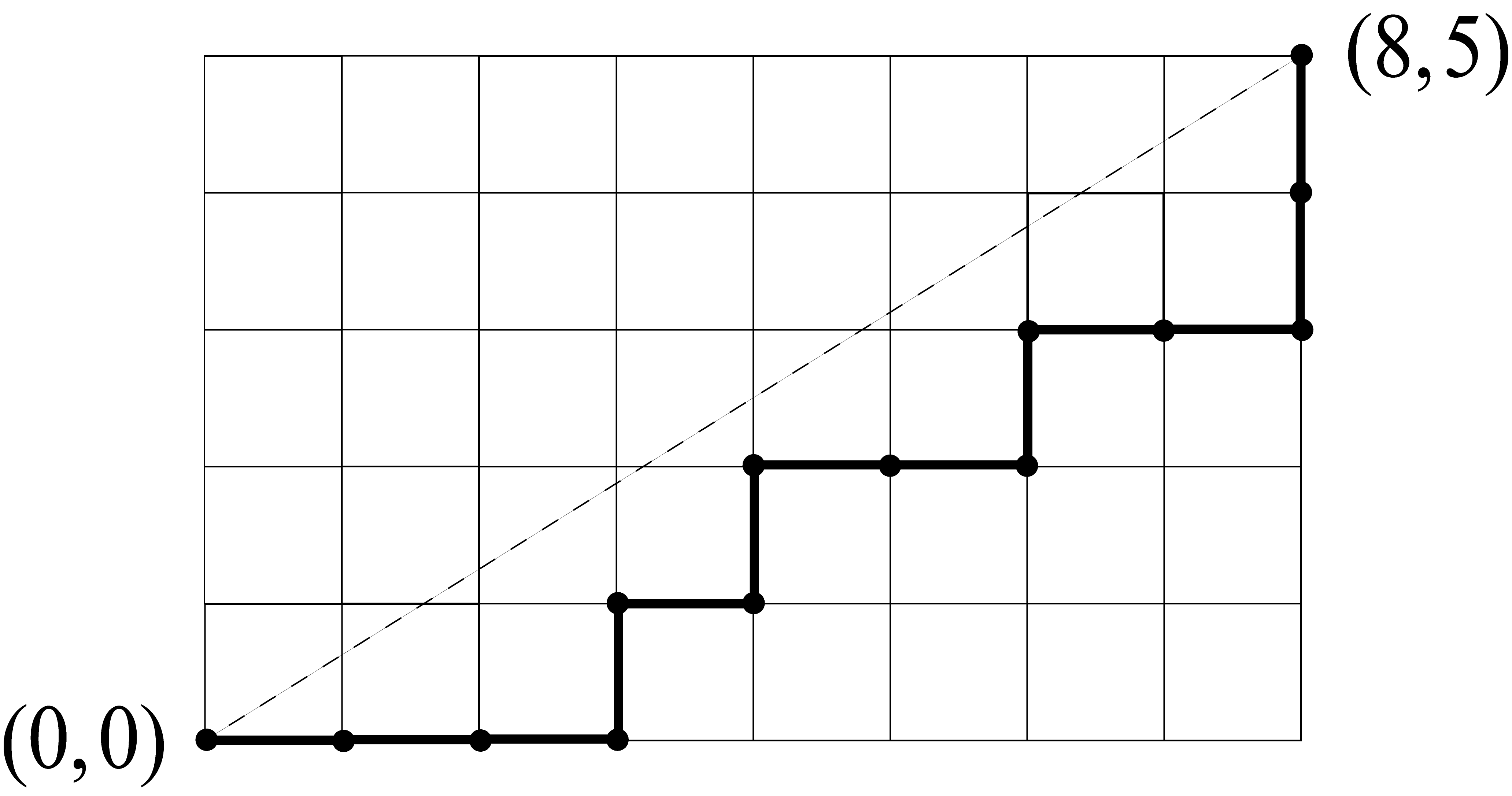}
	\caption{A $(8,5)$-rational Dyck path.} \label{fig:rational Dyck path}
\end{figure}

The classical theory of total positivity, introduced by Gantmacher, Krein, and Schoenberg in the 1930's, has been recently revitalized as a result of the many connections it has with Lusztig's theory of canonical bases for quantum groups (see \cite{gL98} and references therein). In particular, Lusztig extended the theory of total positivity by introducing the totally nonnegative part of a real flag variety. Consequently, an exploration of the combinatorial structure of the totally nonnegative part of the Grassmannian, denoted here by $(\text{Gr}_{d,n})_{\ge 0}$, was initiated by many mathematicians. Motivated by the work of Lusztig and the work of Fomin and Zelevinsky~\cite{FZ99}, Postnikov introduced in~\cite{aP06} positroids as matroids represented by elements of $(\text{Gr}_{d,n})_{\ge 0}$ and showed that they are in bijection with various families of elegant combinatorial objects, including Grassmann necklaces, decorated permutations, \reflectbox{L}-diagrams, and certain classes of plabic graphs (all of them to be introduced later). Positroids and the totally nonnegative Grassmannian have been the focus of much attention lately as they are connected to many mathematical subjects under active investigation including free probability~\cite{ARW16}, soliton solutions to the KP~equation~\cite{KW14}, mirror symmetry~\cite{MR13}, and cluster algebras~\cite{jS06}.

Furthermore, positroids and the totally nonnegative part of the Grassmannian have been also intensively studied in physics because of their applications to compute scattering amplitudes in quantum field theory; readers more inclined to physics might want to consult \cite{ABCGPT16}, which comprises many recent results in this direction. More connections to physics, in particular local space-time physics, can be found in \cite{ABCT11} and \cite{ABCT11a}. In addition, a public package called ``\texttt{positroids}" was added to \emph{Mathematica} (see \cite{jB12}); it includes several tools to create, compute, and visualize decorated permutations, plabic graphs and further objects related to positroids and relevant to computations of scattering amplitudes.

 A $d \times m$ binary matrix is called a \emph{rational Dyck matrix} if its zero entries form a right-justified Young diagram located strictly above the (geometric) main diagonal and anchored in the upper-right corner. If $D = (a_{i,j})$ is a rational Dyck matrix, we shall see in Section~\ref{sec:rational Dyck positroids} that the matrix
 \begin{equation*} \label{UIP generic matrix}
 	\bar{D} =\begin{pmatrix}
 	1  	  		&\dots   &  0        &   0  	 & (-1)^{d-1} a_{d,1} & \dots	& (-1)^{d-1} a_{d,m}  \\
 	\vdots  &\ddots  &\vdots & \vdots & \vdots 	      			& \ddots & \vdots         			\\
 	0  	  		&\dots   &  1        &   0  	 & -a_{2,1}  			& \dots   & -a_{2,m}    		     \\
 	0  	  		&\dots   &  0        &   1  	 & a_{1,1}        			& \dots   & a_{1,m}   
 	\end{pmatrix}
 \end{equation*}
 has all its maximal minors nonnegative. We call the positroids represented by matrices like $\bar{D}$ \emph{rational Dyck positroids}. Notice that the zero entries $a_{i,j}$ of $\bar{D}$ are separated from the nonzero entries $a_{i,j}$ by a rational Dyck path of type $(m,d)$.
 
The main purpose of this paper is to study rational Dyck positroids by establishing combinatorial descriptions for some of the objects parameterizing them. In Section~\ref{sec:rational Dyck positroids} we provide the definitions and results we need to formally describe how to produce rational Dyck positroids from rational Dyck paths. In Section~\ref{sec:grassmann necklace}, a description of the Grassmann necklaces corresponding to rational Dyck positroids is provided. Then, in Section~\ref{sec:decorated permutation}, we study the decorated permutations corresponding to rational Dyck positroids, extending the characterization given in \cite[Theorem~4.4]{CG17}. We also prove that the defining assignment of rational Dyck matrices to rational Dyck positroids is a one-to-one correspondence. In Section~\ref{sec:Le-diagram}, we describe the \reflectbox{L}-diagrams parameterizing rational Dyck positroids, which will yield a simple way to know the dimension of the cell of a rational Dyck positroid in the Grassmann cell decomposition of $(\text{Gr}_{d,m})_{\ge 0}$. In Section~\ref{sec:plabic graph}, we study the classes of plabic graphs parameterizing rational Dyck positroids. Finally, in Section~\ref{sec:positroid polytope}, we provide a description of the matroid polytope of a rational Dyck positroid.

\section{The Positroid Induced by a Rational Dyck Path} \label{sec:rational Dyck positroids}

In this section we formally define \emph{rational Dyck positroids}, which are the objects under study here.

\begin{definition}
	For each pair of nonnegative integers $(m,d)$, a \emph{rational Dyck path} of type $(m,d)$ is a lattice path from $(0,0)$ to $(m,d)$ that only uses unit steps $(1,0)$ or $(0,1)$ and never goes above the diagonal line $y = (d/m)x$.
\end{definition}

When there is no risk of ambiguity, we will abuse notation by referring to a rational Dyck path without specifying the copy of $\rr^2$ in which it is embedded. Let us assign to each rational Dyck path a special binary matrix.

\begin{definition}
	A $d \times m$ binary matrix is called a \emph{rational Dyck matrix} if its zero entries are separated from its one entries by a vertically-reflected rational Dyck path of type $(m,d)$.
\end{definition}

Observe that this definition of a rational Dyck matrix coincides with the description of a rational Dyck matrix given in Section~\ref{sec:intro}. Here is an example of a $5 \times 8$ rational Dyck matrix:
\[
	\begin{pmatrix}
		1 & 1 & 1 & 0 & 0 & 0 & 0 & 0 & 0 \\
		1 & 1 & 1 & 1 & 0  & 0 & 0 & 0 & 0 \\
		1 & 1 & 1 & 1 & 1  & 1 & 0  & 0 & 0 \\
		1 & 1 & 1 & 1 & 1  & 1 & 1  & 1  & 1 \\
		1 & 1 & 1 & 1 & 1  & 1 & 1  & 1  & 1
	\end{pmatrix}.
\]

Let $\mathcal{D}_{d,m}$ denote the set of $d \times m$ rational Dyck matrices. Each rational Dyck path $\mathsf{d}$ of type $(m,d)$ induces the $d \times m$ rational Dyck matrix whose zero entries are separated from its one entries via the vertically-reflected path of $\mathsf{d}$. Our definition of rational Dyck matrix is consistent with the definition of Dyck matrix given in \cite{CG17} when $d = m$. It is well known that standard Dyck matrices are \emph{totally nonnegative}, i.e., all their minors are nonnegative (see, for instance, \cite{ASW52}).

\noindent {\bf Notation:} If $X$ is an $n \times n$ real matrix and $I,J \subseteq [n] := \{1,\dots,n\}$ satisfy $|I| = |J|$, then we let $\Delta_{I,J}(X)$ denote the minor of $X$ determined by the set of rows indexed by $I$ and the set of columns indexed by $J$. Besides, if $Y$ is a $k \times n$ matrix and $K \subseteq [n]$ satisfies $|K| = k$, then we let $\Delta_K(Y)$ denote the maximal minor of $Y$ determined by the set of columns indexed by $K$.

Let $\text{Mat}_{d,m}(\rr)$ denote the set of all $d \times m$ real matrices, and let $\text{Mat}^+_{d,m}(\rr)$ be the subset of $\text{Mat}_{d,m}(\rr)$ consisting of those full-rank matrices with nonnegative maximal minors. Consider the assignment $\phi_{d,m} \colon \text{Mat}_{d,m}(\rr) \to \text{Mat}_{d,d+m}(\rr)$ defined by
\[
	\begin{pmatrix}
		a_{1,1} 		& \dots		 & a_{1,m}     \\
		\vdots 	      & \ddots	   & \vdots     \\
		a_{d-1,1}    & \dots       & a_{d-1,m} \\
		a_{d,1}        & \dots      & a_{d,m}    
	\end{pmatrix} \
	\stackrel{\phi_{d,m}}{\longmapsto} \
	\begin{pmatrix}
		1  	  	 	 &\dots    &  0        &   0  	   & (-1)^{d-1} a_{d,1}   & \dots	& (-1)^{d-1} a_{d,m} \\
		\vdots   &\ddots  &\vdots   & \vdots & \vdots 	      			& \ddots & \vdots \\
		0  	  		&\dots    &  1         &   0  	   & -a_{2,1}  			      & \dots   & -a_{2,m} \\
		0  	  		&\dots    &  0        &   1  	   & a_{1,1}        		   & \dots   & a_{1,m}   
	\end{pmatrix}.\! \!
\]

As the next lemma reveals, the map $\phi_{d,m}$ somehow respects the minors of a given matrix.

\begin{lemma} \cite[Lemma~3.9]{aP06} \label{lem:correspondence between totally nonnegative square and rectangular matrix}\!\footnote{There is a typo in the entries of the matrix $B$ in \cite[Lemma~3.9]{aP06}.}
	If $A \in \emph{Mat}_{d,m}(\rr)$ and $B = \phi_{d,m}(A)$, then
	\[
		\Delta_{I,J}(A) = \Delta_{(d+1 - [d]\setminus I) \cup (d + J)}(B)
	\]
	for all $I \subseteq [d]$ and $J \subseteq [m]$ satisfying $|I| = |J|$.
\end{lemma}

The next lemma, which can be immediately argued by induction, is used in the proof of Proposition~\ref{prop:image of rational Dyck matrices are totally nonnegative}.

\begin{lemma} \label{lem:binary Young matrices are totally nonnegative}
	Every squared binary matrix whose zero entries form a Young diagram anchored in the upper-right corner is totally nonnegative.
\end{lemma}

For a set $S$ and an integer $k$, we let $\binom{S}{k}$ denote the collection consisting of all subsets of $S$ of cardinality $k$, and we call an element of $\binom{S}{k}$ a $k$-\emph{subset} of $S$.

\begin{prop} \label{prop:image of rational Dyck matrices are totally nonnegative}
	The inclusion $\phi_{d,m}(\mathcal{D}_{d,m}) \subseteq \emph{Mat}_{d,d+m}^+(\rr)$ holds.
\end{prop}

\begin{proof}
	Take $D \in \mathcal{D}_{d,m}$, and set $A = \phi_{d,m}(D)$. As $A$ has obviously full rank, it suffices to verify that each of its maximal minors is nonnegative. For $S \in \binom{[d+m]}{d}$ let $A'$ be the submatrix of $A$ determined by the set of columns indexed by $S$. Set $I = S \cap [d]$ and $J = \{j_1, \dots, j_{|S \setminus I|}\} = S \! \setminus \! I$ with $j_1 < \dots < j_{|J|}$. Note that $|J| \le d$. Let $B_J$ be the $d \times d$ matrix whose first $|I|$ columns are all equal to the vector $((-1)^{d-1}, \dots, -1, 1)^t$ and whose $(|I|+k)$-th column is equal to $A_{j_k}$ for $1 \le k \le |J|$. Notice now that Lemma~\ref{lem:binary Young matrices are totally nonnegative} ensures that the matrix $B = (I_d \mid B_J)$ is the image under $\phi_{d,d}$ of a totally nonnegative matrix of size $d$. As Dyck matrices are totally nonnegative, Lemma~\ref{lem:correspondence between totally nonnegative square and rectangular matrix} ensures that every maximal minor of $B$ is nonnegative. In particular, the maximal minor $\det A'$ of $B$ is nonnegative. Hence $A \in \text{Mat}_{d,d+m}^+(\rr)$, as desired.
\end{proof}

Proposition~\ref{prop:image of rational Dyck matrices are totally nonnegative} will allow us to produce positroids from rational Dyck paths. First, we introduce the formal definition of a positroid.

Let $E$ be a finite set, and let $\mathcal{B}$ be a nonempty collection of subsets of $E$. The pair $M = (E, \mathcal{B})$ is a \emph{matroid} if for all $B,B' \in \mathcal{B}$ and $b \in B \setminus B'$, there exists $b' \in B' \setminus B$ such that $(B \setminus \{b\}) \cup \{b'\} \in \mathcal{B}$. The set $E$ is called the \emph{ground set} of $M$, while the elements of $\mathcal{B}$ are called \emph{bases}. Any two bases of $M$ have the same cardinality, which we call the \emph{rank} of $M$.

\begin{definition}
	A rank $d$ matroid $([m], \mathcal{B})$ is called a \emph{positroid} if there exists $A \in \text{Mat}_{d,m}^+(\rr)$ with columns $A_1, \dots, A_m$ such that
	\[
		\mathcal{B} = \big\{B \subseteq [m] \mid \{ A_b \mid b \in B \} \  \text{is a basis for} \ \mathbb{R}^d \big\}.
	\]
	In this case, we say that the matrix $A$ \emph{represents} the positroid $([m], \mathcal{B})$.
\end{definition}

\begin{example} \label{ex:positroid}
	The matrix $A$ below has rank $5$, and all its maximal minors are nonnegative, so it represents a rank $5$ positroid $P$ on the ground set $[12]$.
	\[
		\setcounter{MaxMatrixCols}{20}
		A =\begin{pmatrix}
			1  & 1  & 0 & 0 & 0 & 0 & 0 & 0 &-1 & 0 & 2 & 1 \\
			0  & 0 & 1 & 0 & 0 & 0 & 1  & 2 & 1 & 0 &-1 & 0 \\
			0  & 0 & 0 & 1 & 1  & 0 & 0 & 0 & 0 & 0 & 0 & 0 \\
			0  & 0 & 0 & 0 & 0 & 1 & 1  & 1 & 0 & 0 & 0 & 0 \\
			0  & 0 & 0 & 0 & 0 & 0 & 0 & 0 & 0 & 1 & 1 & 0 
		\end{pmatrix}
	\]
	Notice that $B = \{1,3,4,6,11\}$ is a basis of $P$ as $\Delta_B(A) \neq 0$ while $C = \{1,2,3,4,5\}$ is not a basis because $\Delta_C(A) = 0$.
\end{example}

In virtue of Proposition~\ref{prop:image of rational Dyck matrices are totally nonnegative} every matrix in $\phi_{d,m}(\mathcal{D}_{d,m})$ represents a rank $d$ positroid on the ground set $[d+m]$.

\begin{definition}
	We call a positroid represented by a matrix in $\phi_{d,m}(\mathcal{D}_{d,m})$ a \emph{rational Dyck positroid}, and we denote by $\mathcal{P}_{d,m}$ the set of all rank $d$ rational Dyck positroids on the ground set $[d+m]$.
\end{definition}

Rational Dyck positroids have very nice combinatorial features, which we shall confirm later when we describe their corresponding Grassmann necklaces, decorated permutations, \reflectbox{L}-diagrams, and classes of plabic graphs. For now let us show that rational Dyck positroids are connected matroids.

\begin{definition}
	 A matroid $(E, \mathcal{B})$ is said to be \emph{connected} if for every $b, b' \in E$ there exist $B,B' \in \mathcal{B}$ such that $B' = (B \setminus \{b\}) \cup \{b'\}$.
\end{definition}

\begin{prop}
	Every rational Dyck positroid is connected.
\end{prop}

\begin{proof}
	Let $P$ be a rank $d$ rational Dyck positroid on the ground set $[d+m]$ represented by a matrix $(I_d \mid A)$ in $\phi_{d,m}(\mathcal{D}_{d,m})$. It follows immediately that
	\[
		b \sim b' \ \iff  \text{ there exist } B,B' \in \mathcal{B} \ \text{ such that} \ B' = (B \setminus \{b\}) \cup \{b'\}
	\]
	defines an equivalence relation on $[d+m]$. Notice that $i \sim d+1$ for every $i \in [d]$ as $[d]$ and $([d] \setminus \{i\}) \cup \{d+1\}$ are both bases of $P$. Hence there is an equivalence class $C$ containing the set $[d+1]$. In addition, if $j \in [m]$ and the $j$-th column of $A$ has $n \in [d]$ nonzero entries, then $n \sim d+j$ because both $[d]$ and $([d] \setminus \{n\}) \cup \{d+j\}$ are bases of $P$. Thus, $C = [d+m]$ and the proof follows.
\end{proof}

\section{Grassmann Necklaces} \label{sec:grassmann necklace}

We proceed to introduce the first family of combinatorial objects which can be used to parameterize positroids: the Grassmann necklaces. Then we will describe the Grassmann necklaces corresponding to rational Dyck positroids.

\begin{definition}
	An $n$-tuple $(I_1, \dots, I_n)$ of ordered $d$-subsets of $[n]$ is called a \emph{Grassmann necklace} of type $(d,n)$ if for every $i \in [n]$ the following two conditions hold:
	\begin{itemize}
		\item $i \in I_i$ implies $I_{i+1} = (I_i \setminus \{i\}) \cup \{j\}$ for some $j \in [n]$;
		\item $i \notin I_i$ implies $I_{i+1} = I_i$.
	\end{itemize}
\end{definition}

For $i \in [n]$, the total order $([n], \le_i)$ is defined by $i \le_i \dots \le_i n \le_i 1 \le_i \dots \le_i i-1$. Given a matroid $M = ([n], \mathcal{B})$ of rank $d$, one can define the sequence $\mathcal{I}(M) = (I_1, \dots, I_n)$, where $I_i$ is the lexicographically minimal ordered basis of $M$ with respect to the order $\le_i$. The sequence $\mathcal{I}(M)$ is a Grassmann necklace of type $(d,n)$ (see \cite{aP06}). Moreover, when $M$ is a positroid, we can recover $M$ from $\mathcal{I}(M)$ as we will describe now. For $i \in [n]$, consider the partial order $\preceq_i$ on $\binom{[n]}{d}$ defined in the following way: if $S = \{s_1 \le_i \dots \le_i s_d\}$ and $T = \{t_1 \le_i \dots \le_i t_d\}$ are subsets of $[n]$, then $S \preceq_i T$ if $s_j \le_i t_j$ for each $j \in [d]$.

\begin{theorem}\cite[Theorem~6]{sO11} \label{thm:bijection between Grassmann necklaces and positroids}
	If $\mathcal{I} = (I_1, \dots, I_n)$ is a Grassmann necklace of type $(d,n)$, then
	\[
		\mathcal{B}(\mathcal{I}) = \bigg\{B \in \binom{[n]}{d} \ \bigg{|} \ I_j \preceq_j B \ \text{for each} \ j \in [n] \bigg\}
	\]
	is the collection of bases of a positroid $M(\mathcal{I}) = ([n], \mathcal{B}(\mathcal{I}))$. Moreover, $M(\mathcal{I}(M)) = M$ for all positroids $M$.
\end{theorem}

By Theorem~\ref{thm:bijection between Grassmann necklaces and positroids}, the map $P \mapsto \mathcal{I}(P)$ is a one-to-one correspondence between the set of rank $d$ positroids on the ground set $[n]$ and the set of Grassmann necklaces of type $(d,n)$. For a positroid $P$, we call $\mathcal{I}(P)$ its \emph{corresponding} Grassmann necklace.

\begin{example}
	Let $P$ be the rank $5$ positroid on the ground set $[12]$ introduced in Example~\ref{ex:positroid}, and let $\mathcal{I} = \mathcal{I}(P) = (I_1,\dots, I_{12})$ be the Grassmann necklace of type $(5,12)$ corresponding to $P$. We can see, for instance, that $I_1 = (1,3,4,6,10)$ and $I_{10} = (10,11,12,4,6)$.
\end{example}

Let $D$ be a $d \times m$ rational Dyck matrix, and set $A = (a_{i,j}) = \phi_{d,m}(D)$. Then we define the \emph{set of principal indices} $I_A$ of $A$ to be
\[
	I_A = \{i \in \{d+1, \dots, d+m\} \mid A_i \neq A_{i-1}\},
\]
where $A_i$ denotes the $i$-th column of $A$. In addition, we associate to the matrix $A$ the \emph{weight map} $\omega_A \colon [d+m] \to [d]$ defined by
\[
	\omega_A(j) = \max\{i \mid a_{i,j} \neq 0\}
\]
(notice that there is at least a nonzero entry in each column of $A$). For $d+1 \le j \le d+m$, we observe that the number of nonzero entries in the column~$A_j$ is precisely $\omega_A(j)$.

We are in a position to describe the Grassmann necklaces corresponding to rational Dyck positroids. We will use this description in Section~\ref{sec:positroid polytope} to describe the matroid polytopes of rational Dyck positroids.

\begin{prop} \label{prop:grassmann necklace description}
	Let $P$ be a rational Dyck positroid represented by $A \in \phi_{d,m}(D_{d,m})$, and let $I_A = \{ p_1 < \dots < p_t\}$ and $E_A = \{q_1, \dots, q_u\} = [d] \! \setminus \! \{\omega_A(i) \mid d < i \le d+m\}$. The Grassmann necklace $\mathcal{I}(P) = (I_1, \dots, I_{d+m})$ of $P$, where $I_j = (a^j_1, \dots, a^j_d)$, is characterized as follows.
	\begin{enumerate}
		\item $I_1 = (1,2,\dots,d)$. \vspace{6pt}
		\item If $j \in [d] \! \setminus \! \{1\}$ and $(d-j+1) + |\{p_i \mid \omega_A(p_i) < j-1\}| \ge d$, then $a^j_i = j+i-1$ for $i = 1, \dots, d-j+2$, while $a^j_{d-j+2 + i} = p_{s+i}$ for $s = \max\{k \mid \omega_A(k) \ge j-1\}$ and $i=1, \dots, j-2$. \vspace{6pt}
		\item If $j \in [d] \! \setminus \! \{1\}$ and $(d-j+1) + |\{p_i \mid \omega_A(p_i) < j-1\}| < d$, then $a^j_i = j+i-1$ for $i = 1, \dots, d-j+2$, while $a^j_{d-j+2 + i} = p_{s+i}$ for $s = \max\{k \mid \omega_A(k) \ge j-1\}$ and $i = 1, \dots, t-s$; also, $a^j_{d-j+2+i} = q_{i - (t-s)}$ for $i = t-s+1, \dots, j-2$. \vspace{6pt}
		\item If $j \in \{d+1, \dots, d+m\}$, then $a^j_1 = j$ and $a^j_i = p_{s+i-1}$ for $s = \max\{k \mid p_k \le  j\}$ and $i=2,\dots,t-s+1$ while $a^j_i = q_{i - (t-s+1)}$ for $i = t-s+2, \dots, d+m$.
	\end{enumerate}
\end{prop}

\begin{proof}
	The statement (1) is straightforward. Let us check (2). The lexicographical minimality of $I_j$ with respect to the order $\le_j$ implies that $a^j_i = j+i-1$ for $i = 1, \dots, d-j+1$ as the set $\{A_j, \dots, A_d\}$ consists of $d-j+1$ distinct canonical vectors and so is linearly independent. Also, $A_{d+1} \notin \text{span}(A_j, \dots, A_d)$, yielding $a^j_{d-j+2} = d+1$. Since $(d-j+1) + |\{p_i \mid \omega_A(p_i) < j-1\}| \ge d$, there are enough vectors in $\{A_{p_i} \mid 2 \le i \le t\}$ to complete the basis $I_j$. Here let us make two observations. If $\omega_A(p_i) \ge j-1$, then $A_{p_i}$ is a linear combination of the columns already chosen. As $A_i = A_{p_j}$ when $p_j \le i < p_{j+1}$ the minimality of $I_j$ forces us to complete the basis taking indices in $I_A$. Hence completing $I_j$ amounts to collecting the $j-2$ minimal elements in $I_A$ indexing columns with weights less than $j-1$.
	
	The first part of (3) follows similarly to (2); therefore we will only argue that $a^j_{d-j+2+i} = q_{i - (t-s)}$ for $i = t-s+1, \dots, j-2$. To do so we should take in a minimal way some vectors from $\{A_1, \dots, A_d\}$ to complete $I_j$; it suffices to take the first $j+s-t-2$ indices of $[d]$ which are not in the set $\{\omega_A(a^j_i) \mid \ 1 \le i \le (d-j+2) + (t-s)\}$. Those are precisely the first $j+s-t-2$ smallest elements of $E_A$.
	
	Finally, let us verify (4). Since every column of $A$ is different from the zero vector, $a_1^j = j$. The fact that $a^j_i = p_{s+i-1}$ when $i = 2, \dots, t-s+1$ is an immediate consequence of the minimality of $I_j$; this is because equal columns of $A$ are located consecutively and, for each $i \in [t]$, the column $A_{p_i}$ is located all the way to the left in the block of identical columns it belongs. The equalities $a^j_{d-j+2+i} = q_{i - (t-s)}$ can be argued in the same manner we did in the previous paragraph.
\end{proof}

\section{Decorated Permutations} \label{sec:decorated permutation}

Decorated permutations are generalized permutations that are in natural one-to-one correspondence with positroids. Like Grassmann necklaces, they are combinatorial objects that can be used to parameterize positroids; however, decorated permutations have the extra advantage of offering a more compact parameterization. In this section we characterize the decorated permutations corresponding to rational Dyck positroids.

\begin{definition}
	A \emph{decorated permutation} $\pi$ on $n$ letters is an element $\pi \in S_n$ in which fixed points $j$ are marked either ``clockwise"(denoted by $\pi(j)=\underline{j}$) or ``counterclockwise" (denoted by $\pi(j) = \overline{j}$). A position $j \in [n]$ is called a \emph{weak excedance} of $\pi$ if $j < \pi(j)$ or $\pi(j) = \overline{j}$.
\end{definition}

Following the next recipe, one can assign a decorated permutation $\pi_{\mathcal{I}}$ to each Grassmann necklace $\mathcal{I} = (I_1, \dots, I_n)$:
\begin{enumerate}
	\item if $I_{i+1} = (I_i \setminus \{i\}) \cup \{j\}$, then $\pi_{\mathcal{I}}(j) = i$; \vspace{3pt}
	\item if $I_{i+1} = I_i$ and $i \notin I_i$, then $\pi_\mathcal{I}(i) = \underline{i}$; \vspace{3pt}
	\item if $I_{i+1} = I_i$ and $i \in I_i$, then $\pi_\mathcal{I}(i) = \overline{i}$.
\end{enumerate}
Moreover, the map $\mathcal{I} \mapsto \pi_{\mathcal{I}}$ is a bijection from the set of Grassmann necklaces of type $(d,n)$ to the set of decorated permutations of $n$ letters having $d$ weak excedances. Indeed, it is not hard to verify that the map $\pi \mapsto (I_1, \dots, I_n)$, where
\[
	I_i = \{j \in [n] \mid j \le_i \pi^{-1}(j) \text{ or } \pi(j) = \bar{j}\},
\]
is the inverse of $\mathcal{I} \mapsto \pi_{\mathcal{I}}$. The \emph{corresponding} decorated permutation of a positroid $P$ is $\pi_{\mathcal{I}(P)}$, where $\mathcal{I}(P)$ is the corresponding Grassmann necklace of $P$.

\begin{example}
	Let $P$ be the rank $5$ positroid on the ground set $[12]$ introduced in Example~\ref{ex:positroid}, and let $\mathcal{I} = \mathcal{I}(P) = (I_1,\dots, I_{12})$ be the Grassmann necklace corresponding to $P$. After computing all the entries of $\mathcal{I}$, we can use the recipe described above to verify that $\pi_{\mathcal{I}}$ is, in fact, a permutation which has disjoint cycle decomposition $(1 \ 1\!2 \ 9 \ 2)(3 \ 1\!0 \ 1\!1 \ 7)(4 \ 5)(6 \ 8)$.
\end{example}

For $D \in \mathcal{D}_{d,m}$, set $A = (a_{i,j}) = \phi_{d,m}(D)$. Let $P$ be the positroid represented by $A$, and let $\pi$ be the decorated permutation corresponding to $P$. Note that if we remove one column from $A$ the resulting matrix still has rank $d$. Thus, for each $i \in [d+m]$ the $i$-th entry of the Grassmann necklace corresponding to $P$ does not contain $i-1$. As a result, $\pi$ has no fixed points, which implies that it is a standard permutation. The next proposition, whose proof follows \emph{mutatis mutandis} from that one of \cite[Proposition~4.3]{CG17}, gives an explicit description of the inverse of $\pi$.

\begin{lemma} (cf. \cite[Proposition~4.3]{CG17}) \label{lem:explicity function for decorated permutations}
	If $A, I_A$, $\omega_A$, and $\pi$ are defined as before, then for each $i \in [d+m]$,
	\[
		\pi^{-1}(i) = \fivepf{i+1}{d < i < d+m \text{ and } i+1 \notin I_A}{\omega_A(i)}{d < i \text{ and either } i = d+m \text{ or } i+1 \in I_A}{d+1}{i-1}{1 < i \le d \text{ and } \omega_A(j) \neq i-1 \ \text{for all} \ j \in I_A}{j}{1 < i \le d \text{ and } \{j\} = I_A \cap \omega_A^{-1}(i-1).}
	\]
\end{lemma}

The next proposition generalizes \cite[Theorem~4.4]{CG17}, which states that the disjoint cycle decomposition of $\pi$ consists of only one full cycle.

\begin{prop} \label{prop:decorated permutation of rational Dyck positroids}
	The decorated permutation $\pi$ of a rank $d$ rational Dyck positroid on $[d+m]$ is a $(d+m)$-cycle. Moreover, the set of weak excedances of $\pi$ is $[d]$.
\end{prop}

\begin{proof}
	Let $P$ be a rational Dyck positroid with decorated permutation $\pi$, and let $A \in \phi_{d,m}(\mathcal{D}_{d,m})$ be a matrix representing $P$. It follows from Lemma~\ref{lem:explicity function for decorated permutations} that $\omega_A(i) \le \omega_A(j)$ provided that $\pi(i) = j$ and $j \neq 1$. Let $\sigma$ be a nontrivial cycle of length $\ell$ in the disjoint cycle decomposition of $\pi$. Notice that $\sigma$ cannot fix $1$; otherwise, for $i \in [d+m]$ such that $\sigma(i) \neq i$ we would obtain $\sigma^{j}(i) \neq d+1$ for $j = 1, \dots, \ell$ and so $\omega_A(i) \le \omega_A(\sigma(i)) \le \dots \le \omega_A(\sigma^\ell(i)) = \omega_A(i)$, contradicting the fact that $\sigma$ is nontrivial. Therefore $\pi$ is a $(d+m)$-cycle.
	
	Now let us proceed to show that the set of weak excedances of $\pi$ is precisely $[d]$. To do this, write $\pi^{-1} = (d+1 \ \pi^{-1}(d+1) \ \dots \ (\pi^{-1})^{d+m-2}(d+1) \ \, 1)$. As the map $\omega_A$ fixes each element of $[d]$, the fact that $\omega_A(i) \le \omega_A(j)$ whenever $\pi^{-1}(j) = i$ and $j \neq 1$ ensures that the sequence $(\pi^{-1}(d+1), \dots, (\pi^{-1})^{d+m-2}(d+1))$ contains the elements $2, \dots, d$ in decreasing order. On the other hand, the fact that $\omega_A$ is strictly decreasing on $I_A$ guarantees that the sequence $(\pi^{-1}(d+1), \dots, (\pi^{-1})^{d+m-2}(d+1))$ contains the elements $d+2, \dots, d+m$ in increasing order. Since $\pi$ does not fix any element, its weak excedances are those $j \in [d+m]$ such that $j < \pi(j)$. Because the elements of $[d]$ are the $d$ smallest elements in $[d+m]$ and show increasingly in $\pi = (1 \ \pi(1) \ \dots \ \pi^{d+m-1}(1))$, each element of $[d]$ is a weak excedance of $\pi$. Besides, no element greater than $d$ can be a weak excedance of $\pi$ as $d+1, \dots, d+m$ show decreasingly in $\pi = (1 \ \pi(1) \ \dots \ \pi^{d+m-1}(1))$. Hence $[d]$ is the set of weak excedances of $\pi$.
\end{proof}

Before proving the main theorem of this section, let us collect the next technical result.

\begin{lemma} \label{lem:geometric inequality}
	If $A \in \phi_{d,m}(\mathcal{D}_{d,m})$, then for each $j \in \{d+1, \dots, d+m\}$ the following inequality holds:
	\[
		\omega_A(j) \ge \frac dm (d + m - j + 1).
	\]
\end{lemma}

\begin{proof}
	Let $A = (I_d \mid A')$, and place $A$ into $\rr^2$ in such a way that the rational Dyck path $\mathsf{d}$ of type $(m,d)$ separating the zero and nonzero entries of $A'$ goes from $T = (0,0)$ to $R = (m,d)$. Let $S = (0,d)$ and let $S'$ and $T'$ be the intersection of the vertical line passing through the right endpoint of the $(j-d)$-th horizontal step of $\mathsf{d}$ with the segments $\overline{SR}$ and $\overline{TR}$, respectively. This description is illustrated in the picture below.
	\begin{figure}[h]
		\centering
		\includegraphics[width = 9cm]{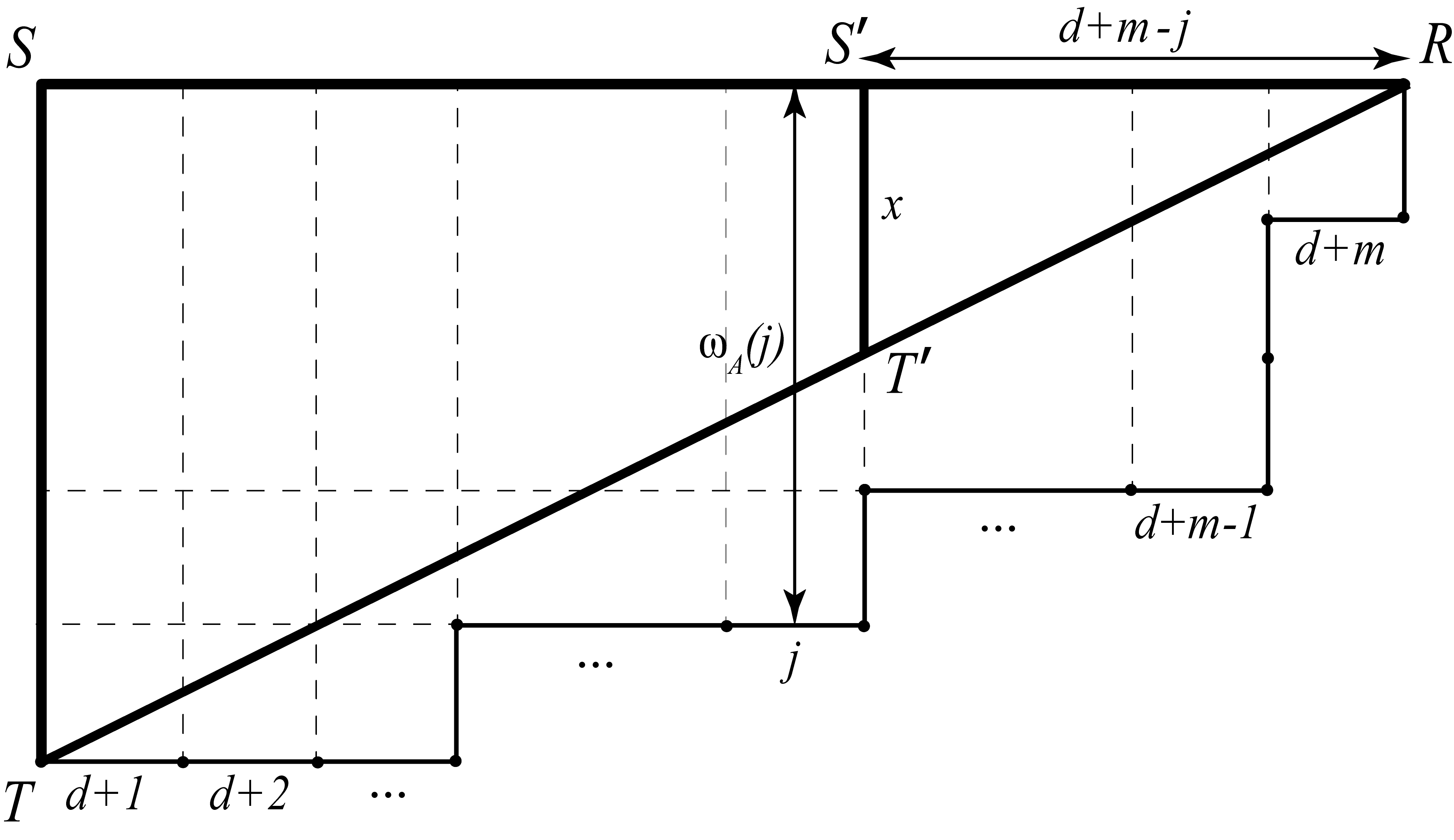}
		\label{fig:geometric inequality}
	\end{figure} 
	
	\noindent If $x$ is the length of the segment $S'T'$, then the similarity of the triangles $RST$ and $RS'T'$ implies that $x/d = (d+m-j)/m$. As the rational Dyck path $\mathsf{d}$ never goes above the diagonal line $y = (d/m)x$, we obtain that
	\[
		\omega_A(j) \ge x + \frac dm = \frac dm (d + m - j + 1),
	\]
	and the lemma follows.
\end{proof}

\begin{theorem}
	There is a bijection between the set of rational Dyck paths of type $(m,d)$ and the set of rank $d$ rational Dyck positroids on the ground set $[d+m]$.
\end{theorem}

\begin{proof}
	Identify the set of rational Dyck paths of type $(m,d)$ with $\mathcal{D}_{d,m}$. Let $\alpha \colon \mathcal{D}_{d,m} \to \mathcal{P}_{d,m}$ be the map assigning to each rational Dyck path of type $(m,d)$ its corresponding rational Dyck positroid via the map $\phi_{d,m}$ in Lemma~\ref{lem:correspondence between totally nonnegative square and rectangular matrix}. We will find a map $\beta \colon \mathcal{P}_{d,m} \to \mathcal{D}_{d,m}$ which is a left inverse of $\alpha$. Let $P$ be a positroid in $\mathcal{P}_{d,m}$ with corresponding decorated permutation $\pi$ such that $\pi^{-1} = (i_1 \ \dots \ i_{d+m})$, where $i_1 = d+1$. Define $\beta(P)$ to be the lattice path $(s_1, \dots, s_{d+m})$ where $s_j = (1,0)$ if $i_j \in \{d+1, \dots, d+m\}$ and $s_j = (0,1)$ if $i_j \in [d]$. Showing that $\beta$ is well defined, i.e., that $\beta(P)$ is a rational Dyck path, will be the fundamental part of this proof.
	
	 Let $P$ be as in the previous paragraph, and let $A \in \phi_{d,m}(\mathcal{D}_{d,m})$ be a matrix representing $P$. For $j = 1,\dots, d+m$, set $S_j = s_1 + \dots + s_j$, and take $\ell_j$ to be the slope of the line determined by the points $(0,0)$ and $S_j$. As $\beta(P)$ consists of $d$ vertical unit steps and $m$ horizontal unit steps, it is a lattice path from $(0,0)$ to $(m,d)$. Suppose, by way of contradiction, that $\beta(P)$ is not a rational Dyck path. Then there exists a minimum $n \in [d+m-1]$ such that $\ell_n > d/m$. The minimality of $n$ implies that $s_n = (0,1)$. Take
	\[
		k = \max\{j \in [d+m] \mid j \ge n \ \text{and} \ s_{j'} = (0,1) \ \text{for all} \ n \le j' \le j\}.
	\]
	As $s_k = (0,1)$ and the elements labeling vertical steps of $\beta(P)$, namely $1,\dots,d$, show decreasingly in $(i_1, \dots, i_{d+m})$, the number of vertical steps in $\{s_1, \dots, s_k\}$ is $d - i_k + 1$. In addition, as $s_{k+1} = (1,0)$ and the elements labeling horizontal steps, namely $d+1, \dots, d+m$, show increasingly in $(i_1, \dots, i_{d+m})$, there are $i_{k+1} - d - 1 = \pi^{-1}(i_k) - d - 1$ horizontal steps in $\{s_1, \dots, s_k\}$. The fact that $\ell_n \le \ell_k$ now implies
	\begin{equation} \label{eq:slope equation}
		 \frac dm < \ell_k = \frac{d - i_k + 1}{\pi^{-1}(i_k) - d - 1}.
	\end{equation}
	After applying some algebraic manipulations to the inequality \eqref{eq:slope equation}, one finds that
	\begin{equation} \label{eq:eq1 after algebraic manipulations}
		\frac dm \big(d+m - \pi^{-1}(i_k) + 1 \big) > i_k - 1.
	\end{equation}
	Since $s_{k+1} = (1,0)$ and $i_{k+1} = \pi^{-1}(i_k)$, we have $d+1 \le \pi^{-1}(i_k) \le d+m$; thus, an application of Lemma~\ref{lem:geometric inequality} yields
	\begin{equation} \label{eq:triangular inequality}
		\omega_A(\pi^{-1}(i_k)) \ge \frac dm \big(d + m - \pi^{-1}(i_k) + 1 \big).
	\end{equation}
	On the other hand, the fact that $i_k \in [d]$, along with Lemma~\ref{lem:explicity function for decorated permutations}, implies that
	\begin{equation} \label{eq:weight of i_k and its inverse}
		i_k = \omega_A(i_k) = \omega_A(\pi^{-1}(i_k)) + 1.
	\end{equation}
	 Now we combine \eqref{eq:eq1 after algebraic manipulations}, \eqref{eq:triangular inequality}, and \eqref{eq:weight of i_k and its inverse} to obtain
	\[
		i_k - 1 = \omega_A(\pi^{-1}(i_k)) \ge \frac dm \big(d+m - \pi^{-1}(i_k) + 1\big) > i_k - 1,
	\]
	which is a contradiction. Hence $\beta(P)$ is indeed a rational Dyck path and, therefore, $\beta$ is a well-defined function.

	To verify that $\beta$ is a left inverse of $\alpha$, take $\mathsf{d} = (d_1, \dots, d_{d+m}) \in \{(1,0), (0,1)\}^{d+m}$ to be a rational Dyck path in $\mathcal{D}_{d,m}$ such that $\alpha(\mathsf{d}) = P$, and set $\mathsf{d}' = \beta(P) = (d'_1, \dots, d'_{d+m})$. As before, let $\pi^{-1} = (i_1 \ \dots \ i_{d+m})$, where $i_1 = d+1$. Let us verify that $\mathsf{d}' = \mathsf{d}$. Suppose inductively that $d'_j = d_j$ for each $j \in [d+m-1]$ (notice that $d'_1 = d_1 = (1,0)$). As $j < d+m$, it follows that $i_j \neq 1$. Assume first that $d_j = (1,0)$, which means that $i_j \in \{d+1, \dots, d+m\}$. If $\pi^{-1}(i_j) = i_j + 1$, then $d_{j+1}$ is also a horizontal step by Lemma \ref{lem:explicity function for decorated permutations}. Also, the fact that $i_{j+1} = i_j + 1 \in \{d+1, \dots, d+m\}$ guarantees that $d'_{j+1}$ is a horizontal step too. On the other hand, if $\pi^{-1}(i_j) \neq i_j + 1$ (which means that $\pi^{-1}(i_j) = \omega_A(i_j)$), then $d_{j+1}$ is vertical by Lemma~\ref{lem:explicity function for decorated permutations}. In addition, $\pi^{-1}(i_j) = \omega_A(i_j) \le d$ implies that $d'_{j+1}$ is also a vertical step. Hence $d'_{j+1} = d_{j+1}$. In a similar fashion the reader can verify that $d'_{j+1} = d_{j+1}$ when $d_j = (0,1)$. The fact that $\alpha$ has a left inverse function, along with $|\mathcal{D}_{d,m}| \ge |\mathcal{P}_{d,m}|$, yields that $\alpha$ is a bijection.
\end{proof}

\begin{cor}
	The number of rank $d$ rational Dyck positroids on the ground set $[d+m]$ equals the rational Catalan number $\emph{Cat}(m,d)$. In particular, there are $\frac 1{d+m} \binom{d+m}{d}$ rank $d$ rational Dyck positroids on the ground set $[d+m]$ when $\gcd(d,m) = 1$.
\end{cor}

The following proposition provides a very simple way to compute the decorated permutation of a given rational Dyck positroid directly from its rational Dyck path. We will omit the proof as it follows with no substantial changes the proof of \cite[Proposition~5.1]{CG17}.

\begin{prop} \label{prop:decoding decorated permutation}
	Let $\mathsf{d}$ be a rational Dyck path of type $(m,d)$. Labeling the $d$ vertical steps of $\mathsf{d}$ from top to bottom in increasing order with $1, \dots, d$ and the $m$ horizontal steps from left to right in increasing order with $d+1, \dots, d+m$, we obtain the decorated permutation of the rational Dyck positroid induced by $\mathsf{d}$ by reading the step labels of $\mathsf{d}$ in southwest direction.
\end{prop}

\begin{example}
	Let $P$ be the rational Dyck positroid induced by the rational Dyck path $\mathsf{d}$ of type $(8,5)$ illustrated in Figure~\ref{fig:permutation from rational Dyck path}. The path $\mathsf{d}$ is labeled as indicated in Proposition~\ref{prop:decoding decorated permutation}. Therefore the decorated permutation of $P$ is $\pi = (1 \ 2 \ 1\!3 \ 1\!2 \ 3 \ 1\!1 \ 1\!0 \ 4 \ 9 \ 5 \ 8 \ 7 \ 6)$, which is obtained by reading the labels of $\mathsf{d}$ in southwest direction.
	\begin{figure}[h]
		\centering
		\includegraphics[width = 6.0cm]{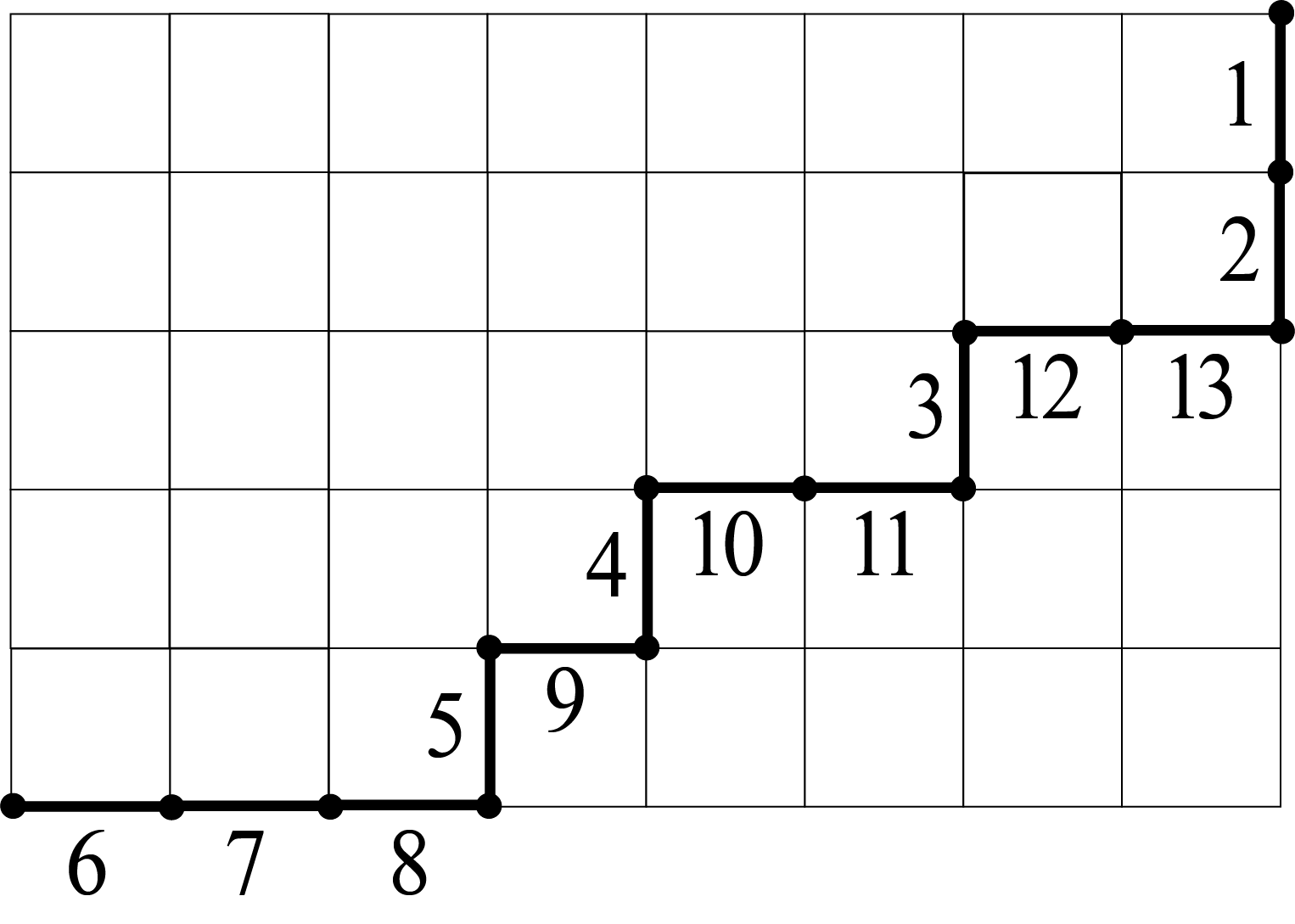}
		\caption{A rational Dyck path of type $(8,5)$ encoding the decorated permutation of the rational Dyck positroid it induces.}
		\label{fig:permutation from rational Dyck path}
	\end{figure}
\end{example}

\section{Le-diagrams} \label{sec:Le-diagram}

In this section we describe the \reflectbox{L}-diagrams corresponding to rational Dyck positroids. These combinatorial objects not only are in bijection with positroids but also encode the dimension of the Grassmann cells containing the positroids they parameterize.

\begin{definition}
	Let $d$ and $m$ be positive integers, and let $Y_\lambda$ be the Young diagram associated to a given partition $\lambda$ contained in a $d \times m$ rectangle. A \emph{\reflectbox{\emph{L}}-diagram} (or \emph{Le-diagram}) $L$ of shape $\lambda$ and type $(d,d+m)$ is obtained by filling the boxes of $Y_\lambda$ with zeros and pluses so that no zero entry has simultaneously a plus entry above it in the same column and a plus entry to its left in the same row.
\end{definition}

With notation as in the above definition, the southeast border of $Y_\lambda$ determines a path of length $d+m$ from the northeast to the southwest corner of the $d \times m$ rectangle; we call such a path the \emph{boundary path} of $L$.


It is well known that there is a natural bijection $\Phi$ from the set of \reflectbox{L}-diagrams of type $(d,d+m)$ to the set of decorated permutations on $[d+m]$ having exactly $d$ excedances (see \cite[Section~20]{aP06}). Thus, \reflectbox{L}-diagrams of type $(d,d+m)$ also parameterize rank $d$ positroids on the ground set $[d+m]$. Moreover, if $\Phi \colon L \to \pi$ and we label the steps of the boundary path of $L$ in southwest direction, then $i \in [d+m]$ labels a vertical step of the boundary path of $L$ if and only if $i$ is a weak excedance of $\pi$ (see \cite[Lemma~5]{SW07}).

\begin{example} \label{ex:Le-diagram}
	The picture below shows a \reflectbox{L}-diagram $L$ of type $(5,12)$ with its boundary path highlighted. The decorated permutation $\Phi(L)$ is $(1 \ 1\!2 \ 9 \ 2)(3 \ 1\!0 \ 1\!1 \ 7)(4 \ 5)(6 \ 8)$ and, therefore, $L$ parameterizes the positroid $P$ introduced in Example~\ref{ex:positroid}.
	\begin{figure}[h]
		\centering
		\includegraphics[width = 6.0cm]{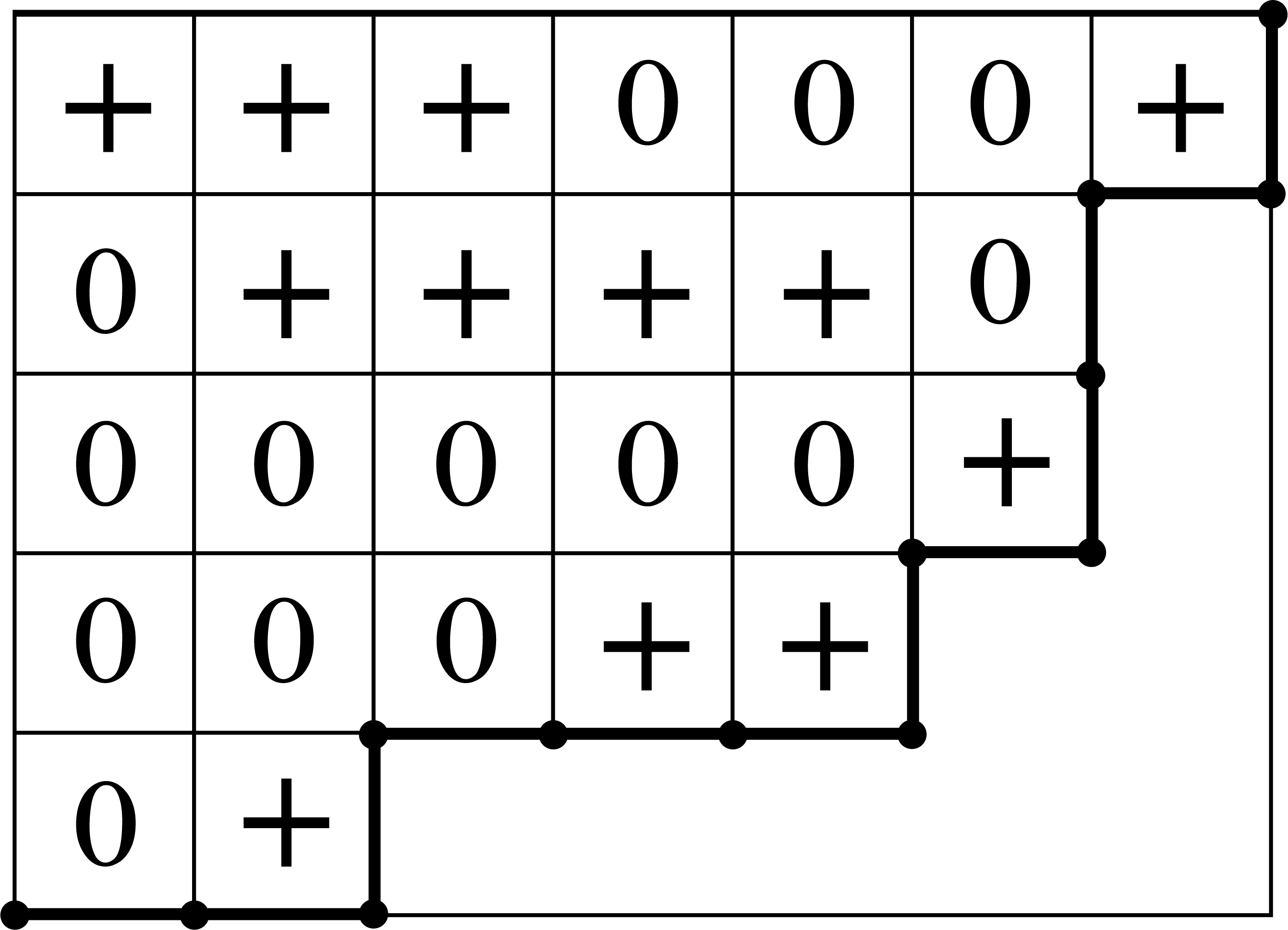}
		\caption{A Le-diagram of type $(5,12)$ and shape $\lambda = (7,6,6,5,2)$.}
		\label{fig:Le-diagram}
	\end{figure}
\end{example}

Let $\lambda$ be a partition, and let $Y_\lambda$ be the Young diagram associated to $\lambda$. We call a \emph{pipe dream} of shape $\lambda$ to a tiling of $Y_\lambda$ by elbow joints $\,\raisebox{-0.1 \height}{\includegraphics[width=0.5cm]{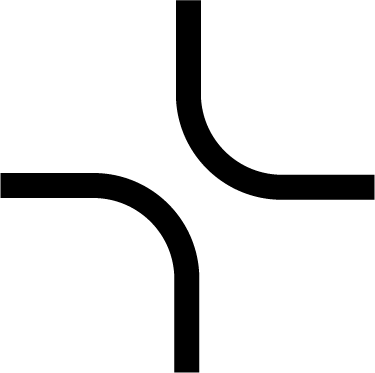}}$ and crosses $\,\raisebox{-0.1 \height}{\includegraphics[width=0.5cm]{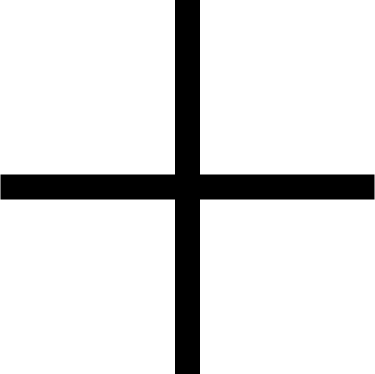}}$. The next lemma yields a method (illustrated in Figure~\ref{fig:Le-diagram and pipe diagram}) to find the decorated permutation $\pi = \Phi(L)$ corresponding to a positroid directly from its \reflectbox{L}-diagram.
 
\begin{lemma}\cite[Lemma~4.8]{ARW16} \label{lem:decorated permutation from Le-diagram}
	Let $L$ be the \reflectbox{\emph{L}}-diagram corresponding to a rank $d$ positroid $P$ on $[d+m]$. We can compute the decorated permutation $\pi$ of $P$ as follows.
	\begin{enumerate}
		\item Replace the pluses in the \reflectbox{\emph{L}}-diagram $L$ with elbow joints $\,\raisebox{-0.1 \height}{\includegraphics[width=0.5cm]{LeDiagramElbow.png}}$ and the zeros in $L$ with crosses $\,\raisebox{-0.1 \height}{\includegraphics[width=0.5cm]{LeDiagramPlus.png}}$ to obtain a pipe dream.
		\item Label the steps of the boundary path with $1, \dots, d+m$ in southwest direction, and then label the edges of the north and west border of $Y_\lambda$ also with $1,\dots, d+m$ in such a way that labels of opposite border steps coincide.
		\item Set $\pi(i) = j$ if the pipe starting at the step labeled by $i$ in the northwest border ends at the step labeled by $j$ in the boundary path. If $\pi$ fixes $j$ write $\pi(j) = \underline{j}$ (resp., $\pi(j) = \overline{j}$) if $j$ labels a horizontal (resp., vertical) step of the boundary path.
	\end{enumerate}
\end{lemma}

We have seen in Proposition~\ref{prop:decorated permutation of rational Dyck positroids} that the decorated permutation of a rank $d$ rational Dyck positroid on the ground set $[d+m]$ has $[d]$ as its set of weak excedances. Thus, the \reflectbox{L}-diagram of a rational Dyck positroid has a rectangular shape, namely $m^d$. The next description, which gives a complete characterization of the \reflectbox{L}-diagrams parameterizing rational Dyck positroids, has been proved in \cite{CG17} for the case $d=m$. However, the argument for proving the general case is basically a reproduction of the case $d=m$ and, therefore, we decided to omit it.

\begin{prop} \label{prop:Le-Diagrams of rational Dyck positroids}
	A \reflectbox{\emph{L}}-diagram $L$ of type $(d,d+m)$ corresponds to a rank $d$ rational Dyck positroid on the ground set $[d+m]$ if and only if its shape $\lambda$ is the full $d \times m$ rectangle and $L$ satisfies the following two conditions:
	\begin{enumerate}
		\item every column has exactly on plus except the last one that has $d$ pluses;
		\vspace{3pt}
		\item the horizontal unit steps right below the bottom-most plus of each column are the horizontal steps of a horizontally-reflected rational Dyck path of type $(m,d)$ (see Figure~\ref{fig:Le-diagram and pipe diagram} below).
	\end{enumerate}
\end{prop}

\begin{example}
	Let $P$ be the rank $5$ rational Dyck positroid on the ground set $[13]$ having decorated permutation $\pi = (1 \ 2 \ 1\!3 \ 1\!2 \ 3 \ 1\!1 \ 1\!0 \ 4 \ 9 \ 5 \ 8 \ 7 \ 6)$. The following picture showing the \reflectbox{L}-diagram corresponding to $P$ along with its associated pipe dream sheds light upon the recipe described in Lemma~\ref{lem:decorated permutation from Le-diagram}.
	
	\begin{figure}[h]
		\centering
		\raisebox{0.13 \height}{\includegraphics[width = 6.0cm]{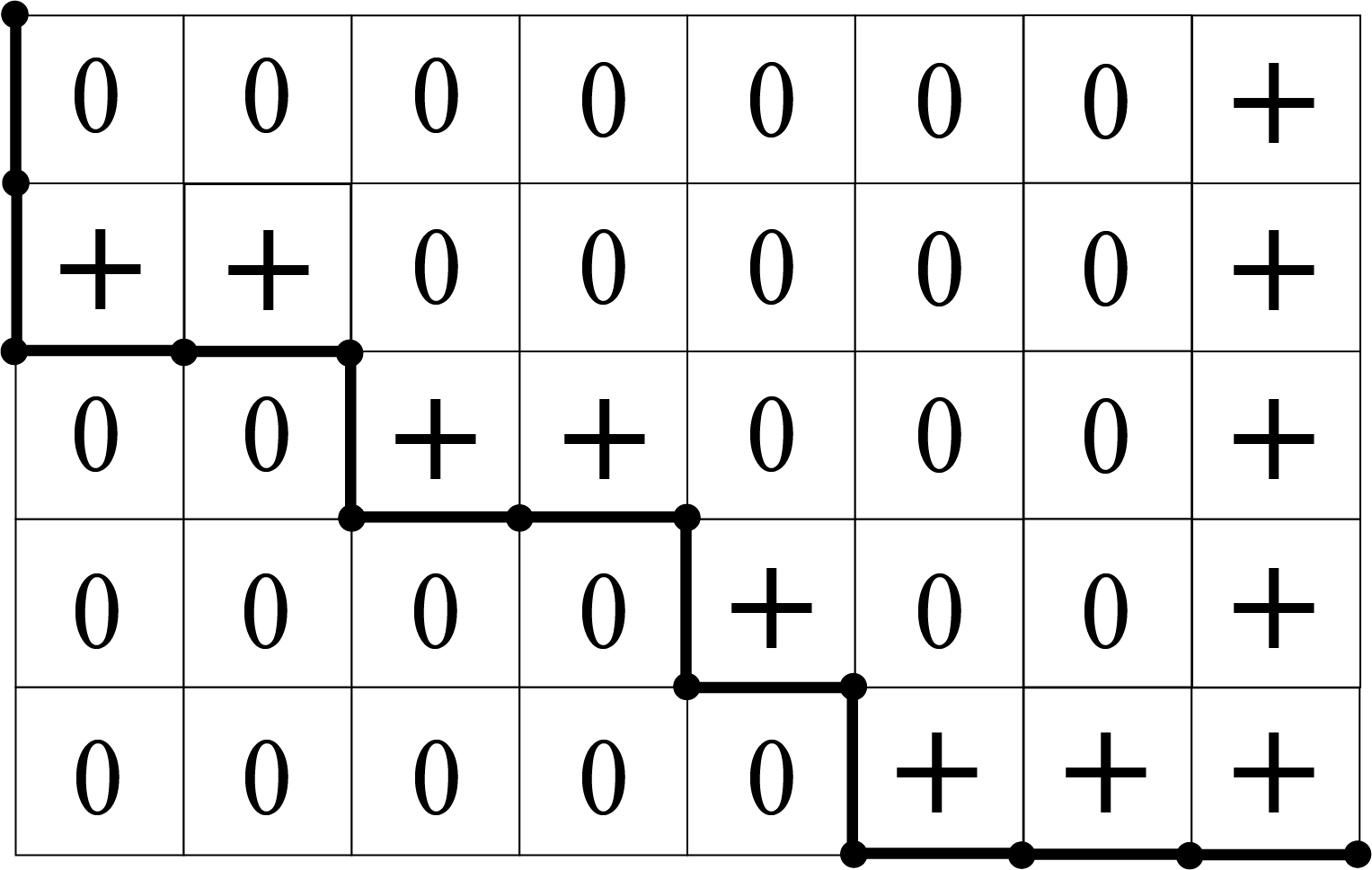}} \hspace{2.1cm}
		\includegraphics[width = 6.4cm]{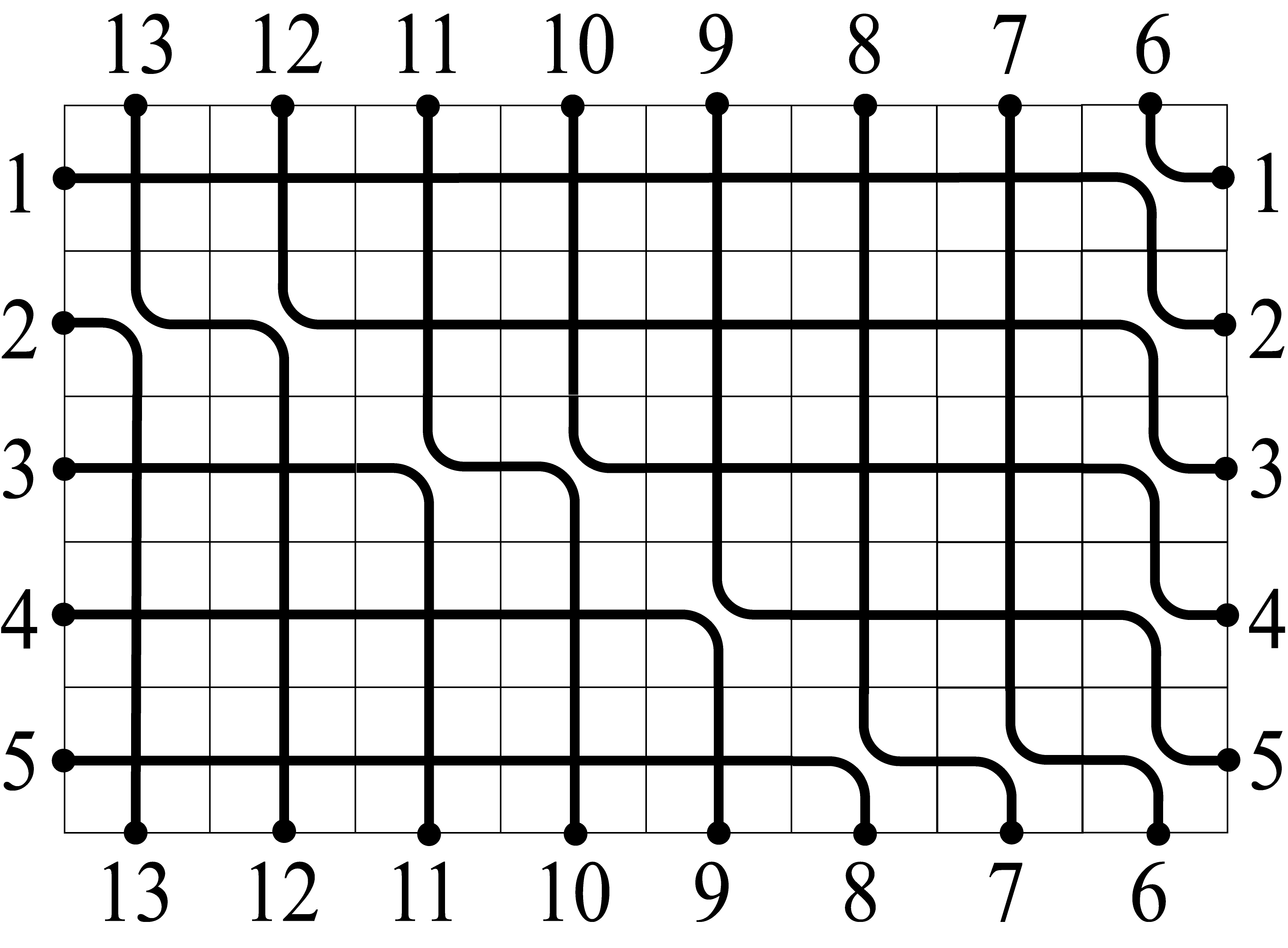}
		\caption{The Le-diagram of $P$ on the left and the corresponding pipe dream giving rise to $\pi$.}
		\label{fig:Le-diagram and pipe diagram}
	\end{figure}

\end{example}

For $d \le n$, the \emph{real Grassmannian} $\text{Gr}_{d,n}$ is the set of all subspaces of $\rr^n$ of dimension $d$. Each element of the real Grassmannian can be thought of as the row span $[A]$ of a full-rank $d \times n$ real matrix $A$. Therefore $\text{Gr}_{d,n}$ is the quotient of the set comprising all full-rank $d \times n$ real matrices under the left action of $\text{GL}_d(\rr)$. For $[A] \in \text{Gr}_{d,n}$, let $M_{[A]}$ denote the rank $d$ matroid represented by $A$ (the matroid $M_{[A]}$ does not depend on the matrix representation of $[A]$). Every rank $d$ representable matroid $M$ determines a \emph{matroid stratum}
\[
	S_M := \{[A] \in \text{Gr}_{d,n} \mid M_A = M \}.
\]
The nonempty matroid strata $S_M$ induce a subdivision of $\text{Gr}_{d,n}$, which is known as \emph{Gelfand-Serganova strata}. The \emph{totally nonnegative Grassmannian} $(\text{Gr}_{d,n})_{\ge 0}$ is defined as the quotient
\[
	(\text{Gr}_{d,n})_{\ge 0} := \text{GL}^+_d(\rr) \! \setminus \! \text{Mat}^+_{d,n}(\rr),
\]
where $\text{GL}_d^+(\rr)$ is the set of all $d \times d$ real matrices having positive determinant. Each rank $d$ positroid $P$ on the ground set $[n]$ determines a \emph{positroid cell} $S^+_P = S_P \cap (\text{Gr}_{d,n})_{\ge 0}$ inside $(\text{Gr}_{d,n})_{\ge 0}$, where $S_P$ is the matroid stratum of $P$ in $\text{Gr}_{d,n}$. Positroid cells are, therefore, naturally indexed by \reflectbox{L}-diagrams. It is well known that the number of pluses inside each indexing \reflectbox{L}-diagram equals the dimension of its positroid cell (see \cite{gL98} and \cite{aP06}). The next corollary follows immediately from Proposition~\ref{prop:Le-Diagrams of rational Dyck positroids}.
	
\begin{cor}
	The positroid cell parameterized by a rank $d$ rational Dyck positroid on the ground set $[d+m]$ inside the corresponding Grassmannian cell complex has dimension $d+m-1$.
\end{cor}

\section{Plabic Graphs} \label{sec:plabic graph}

Let us now proceed to characterize the move-equivalence classes of plabic graphs (up to homotopy) corresponding to rational Dyck positroids.

\begin{definition}
	A \emph{plabic graph} is an undirected graph $G$ drawn inside a disk $D$ (up to homotopy) which has a finite number of vertices on the boundary of $D$. Each of the remaining vertices of $G$ is strictly inside the disk and colored either black or white. Each vertex on the boundary of $D$ is incident to a single edge.
\end{definition}

With notation as in the above definition, the vertices in the interior of $D$ are called \emph{internal vertices} of $G$ while the vertices on the boundary of $D$ are called \emph{boundary vertices} of $G$. In the context of this paper the boundary vertices of $G$ are always going to be clockwise labeled starting by $1$. Also, every plabic graph here is assumed to be \emph{leafless} (there are no internal vertices of degree one) and without isolated components. For the rest of this section, let $G$ denote a plabic graph with $n$ boundary vertices.

A \emph{perfect orientation} $\mathcal{O}$ of $G$ is a choice of directions for every edge of $G$ in such a way that black vertices have outdegree one and white vertices have indegree one. If $G$ admits a perfect orientation $\mathcal{O}$, we call $G$ \emph{perfectly orientable} and let $G_{\mathcal{O}}$ denote the directed graph on $G$ determined by $\mathcal{O}$. A boundary vertex $v$ of an oriented plabic graph $G_{\mathcal{O}}$ is a \emph{source} (resp., \emph{sink}) if $v$ has indegree (resp., outdegree) zero. The set of boundary vertices that are sources (resp., sinks) of $G_{\mathcal{O}}$ is denoted by $I_{\mathcal{O}}$ (resp., $\bar{I}_{\mathcal{O}}$). It is known that any two perfect orientations of the same plabic graph $G$ have source sets of the same size
\begin{equation} \label{eq:type of a plabic graph}
	d := \frac 12 \bigg(n + \sum_{v \text{ black}} \big(\deg(v) - 2\big) + \sum_{v \text{ white}}\big(2 - \deg(v)\big)\bigg).
\end{equation}
The \emph{type} of $G$ is defined to be $(d,m)$. See \cite{aP06} for more details.

The next local transformations will partition the set of plabic graphs into equivalence classes. We will see later that such a set of equivalence classes is in one-to-one correspondence with the set of positroids.

(M1) {\bf Square move:} If $G$ has a square consisting of four trivalent vertices whose colors alternate, then the colors of these four vertices can be simultaneously switched.
\begin{figure}[h]
	\centering
	\includegraphics[width = 6cm]{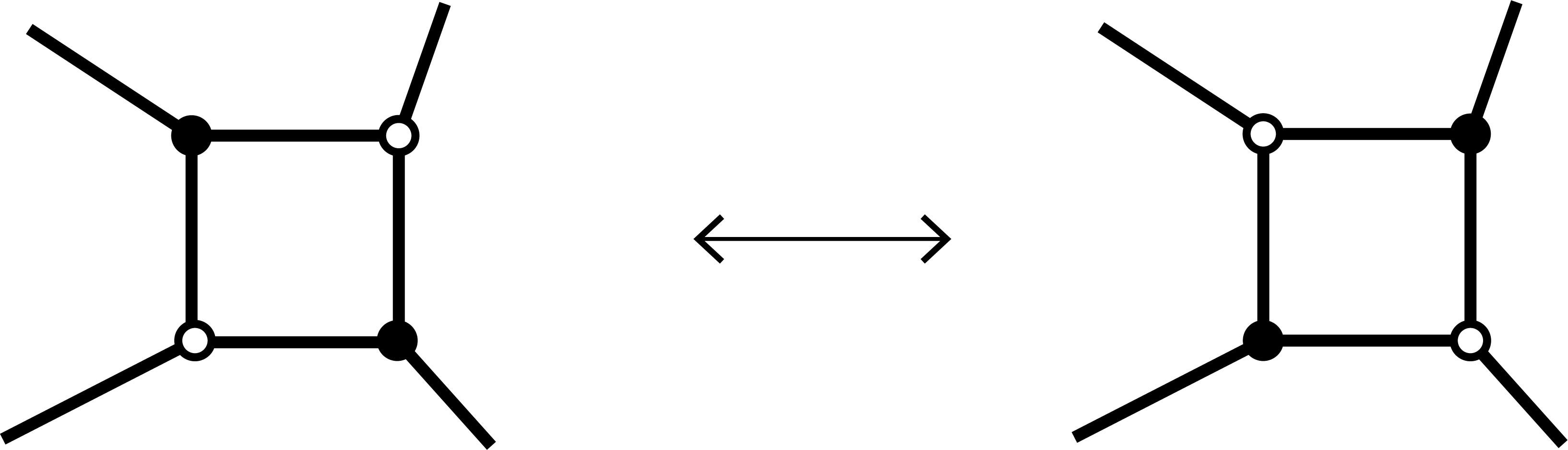}
\end{figure}

(M2) {\bf Unicolored edge contraction/uncontraction:} If $G$ contains two adjacent vertices of the same color, then any edge joining these two vertices can be contracted into a single vertex with the same color of the two initial vertices. Conversely, a given vertex of $G$ can be uncontracted into an edge joining vertices of the same color as the given vertex.
\begin{figure}[h]
	\centering
	\includegraphics[width = 6cm]{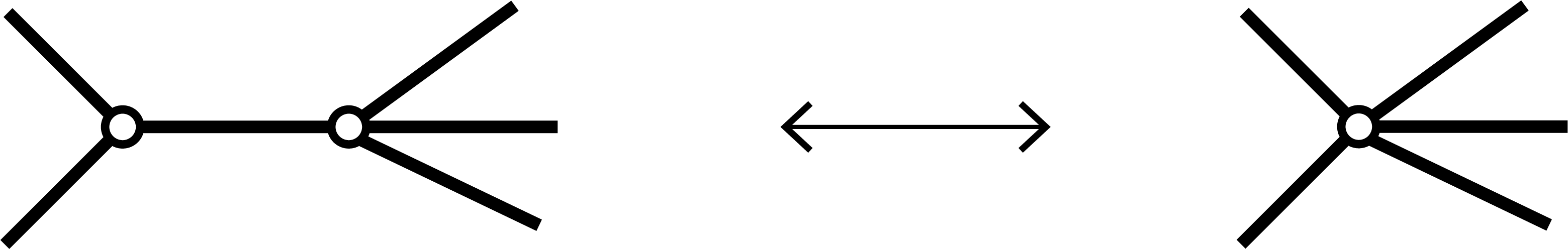}
\end{figure}

(M3) {\bf Middle vertex insertion/removal:} If $G$ contains a vertex of degree $2$, then this vertex can be removed and its incident edges can be glued together. Conversely, a vertex (of any color) can be inserted in the middle of any edge of $G$.
\begin{figure}[h]
	\centering
	\hspace{0.2cm} \includegraphics[width = 6cm]{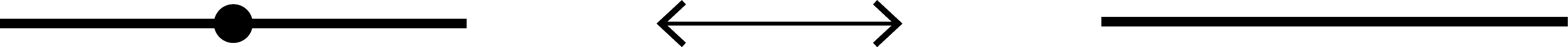}
\end{figure}

(R1) {\bf Parallel edge reduction:} If $G$ contains two trivalent vertices of different colors connected by a pair of parallel edges, then these vertices and edges can be deleted, and the remaining two edges can be glued together.
\begin{figure}[h]
	\centering
	\includegraphics[width = 7cm]{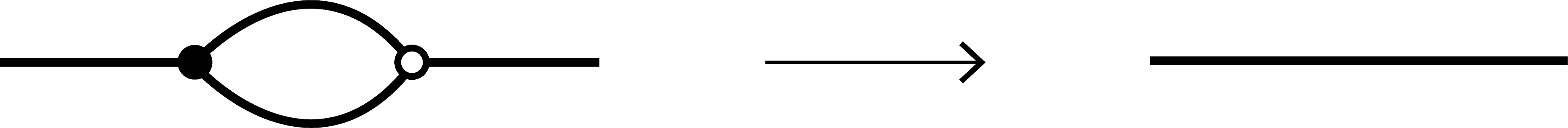}
\end{figure}

Two plabic graphs are called \emph{move-equivalent} if they can be obtained from each other by applying the local transformations (M1), (M2), and (M3); this defines an equivalence relation on the set of plabic graphs. A leafless plabic graph $G$ without isolated components is said to be \emph{reduced} if (R1) cannot be applied to any plabic graph in the move-equivalence class of $G$. Any reduced plabic graph is known to be perfectly orientable.

We define $\pi_G$ by setting $\pi_G(i) = j$ if there is a path in $G$ from the boundary vertex $i$ to the boundary vertex $j$ that turns left (resp., right) at any internal black (resp., white) vertex; such a path is said to follow the ``rules of the road." If $i$ is fixed by $\pi_G$, then we mark $i$ clockwise (resp., counterclockwise) whenever the internal vertex of $G$ adjacent to $i$ is black (resp., white).

\begin{definition}
	For a plabic graph $G$, the trip $\pi_G$ described above is called the \emph{decorated trip permutation} of $G$.
\end{definition}

\begin{theorem} \cite[Theorem~13.4]{aP06}
	Two reduced plabic graphs are move-equivalent if and only if they have the same decorated trip permutation.
\end{theorem}

As Grassmann necklaces, decorated permutations, and \reflectbox{L}-diagrams, move-equivalence classes of plabic graphs of type $(d,n)$ also parameterize rank $d$ positroids on the ground set $[n]$.

\begin{prop} \cite[Section~11]{aP06}
	For each pair of positive integers $d$ and $n$ with $d \le n$ the assignment $G \mapsto P_G = ([n], \mathcal{B}_G)$, where
	\[
		\mathcal{B}_G = \{I_{\mathcal{O}} \mid \mathcal{O} \ \text{is a perfect orientation of} \ G\},
	\]
	is a one-to-one correspondence between move-equivalence classes of perfectly orientable plabic graphs of type $(d,n)$ and rank $d$ positroids on the ground set $[n]$.
\end{prop}

\begin{example}
	Let $P$ be the rank  $5$ positroid on the ground set $[12]$ introduced in Example~\ref{ex:positroid}. The following picture shows an oriented plabic graph $G_{\mathcal{O}}$ corresponding to $P$. The perfect orientation $\mathcal{O}$ gives the basis $I_{\mathcal{O}} = \{1,5,6,8,10\}$ of $P$. We have seen before that the decorated permutation $\pi$ corresponding to $P$ has disjoint cycle decomposition $(1 \ 1\!2 \ 9 \ 2)(3 \ 1\!0 \ 1\!1 \ 7)(4 \ 5)(6 \ 8)$. In particular, $\pi(3) = 10$, which is indicated by the directed path from $3$ to $10$ highlighted in the picture; observe that such a path follows the rules of the road.
	
	\begin{figure}[h]
		\centering
		\includegraphics[width = 5.7cm]{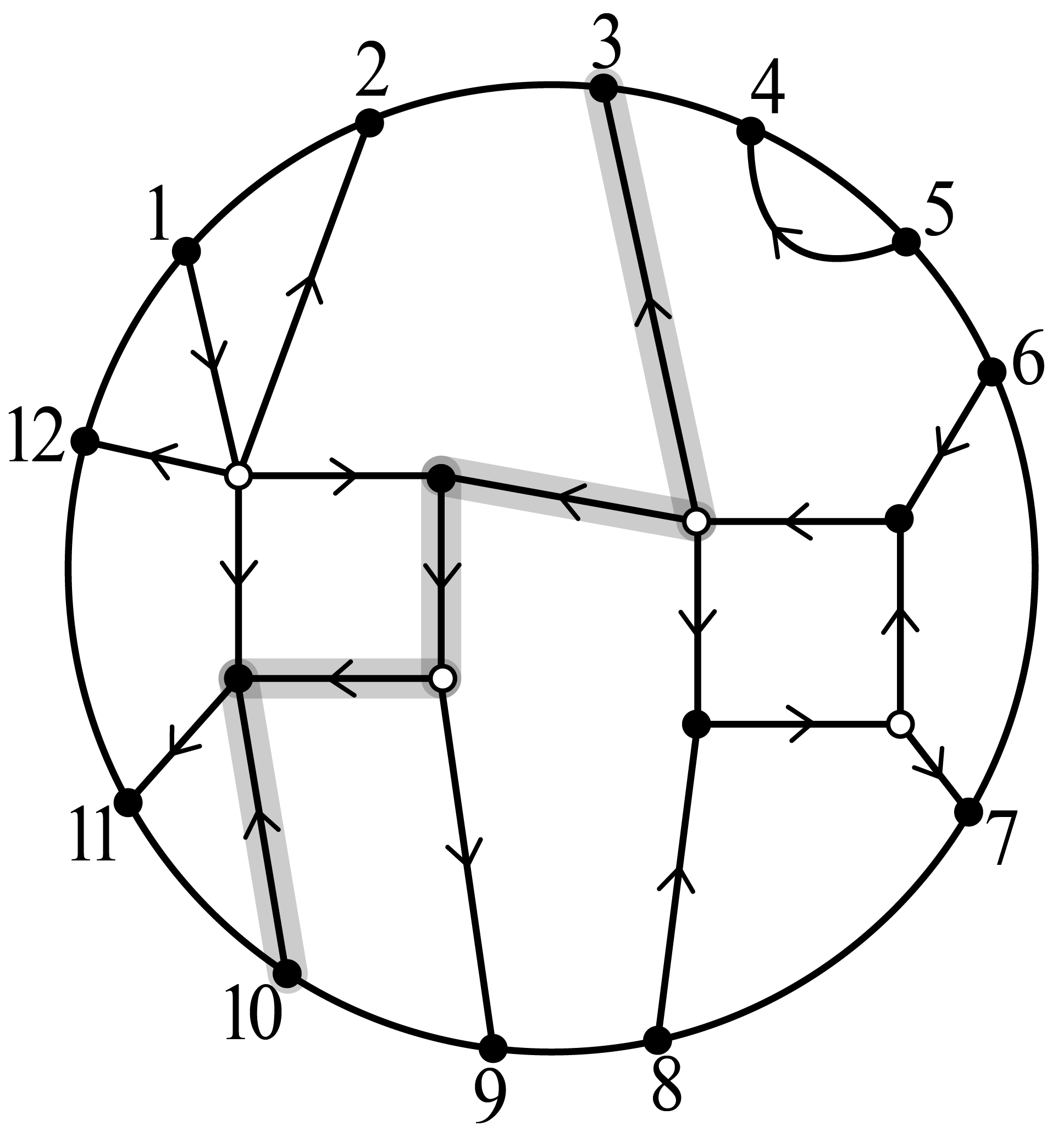}
		\caption{A plabic graph with a perfect orientation.}
		\label{fig:plabic graph}
	\end{figure}
	
\end{example}

We conclude this section characterizing the plabic graphs corresponding to rational Dyck positroids.

\begin{prop} \label{prop:the plabic graph of a rational Dyck positroid}
	A rational Dyck path $\mathsf{d}$ of type $(m,d)$ induces a reduced plabic graph $G_{\mathsf{d}}$ of type $(d,d+m)$ as follows:
	\begin{enumerate}
		\item Draw a circle with $(0,0)$ and $(m,d)$ diametrically opposed, and draw a black (resp., white) vertex in the middle of each vertical (resp., horizontal) step of $\mathsf{d}$.
		\item Draw a horizontal segment from each white vertex to the circle (going east) and label the intersections by $1, \dots, d$ (clockwise). Similarly, draw a vertical segment from each black vertex to the circle (going north) and label the intersections by $d+1, \dots, d+m$ (clockwise).
		\item Finally, join consecutive internal vertices in $\mathsf{d}$ by segments and ignore the initial rational Dyck path $\mathsf{d}$ (see Figure~\ref{fig:plabic graph of a rational Dyck positroid and its reduced representative.}).
	\end{enumerate}
\end{prop}

\begin{proof}
	It follows immediately that the given recipe yields a plabic graph with $d+m$ boundary vertices. To find the type of $G_{\mathsf{d}}$, notice that all internal vertices have degree $3$, except the first internal white vertex and the last internal black vertex on $\mathsf{d}$ (in northeast direction) which have degree $2$. Therefore, using the formula \eqref{eq:type of a plabic graph}, one obtains
	\[
		\frac 12 \bigg(n + \sum_{v \text{ black}} \big(\deg(v) - 2\big) + \sum_{v \text{ white}}\big(2 - \deg(v)\big)\bigg) = \frac 12 \big(d+m + (d-1) - (m-1) \big) = d.
	\]
	Thus, the type of $G_{\mathsf{d}}$ is $(d,d+m)$.
	
	Let us verify now that the graphs in the move-equivalence class of plabic graphs of a rational Dyck positroid are trees. Let $G$ be a plabic graph representing a rational Dyck positroid $P$ of type $(d,d+m)$, and let us check that $G$ is a tree. As any two graphs in the same move-equivalence class of plabic graphs corresponding to $P$ are homotopic, it suffices to assume that $G$ is the representative described in Proposition~\ref{prop:the plabic graph of a rational Dyck positroid}. Suppose, by way of contradiction, that $G$ is not a tree, meaning that it has a cycle consisting of the vertices $v_1, \dots, v_k$ for some $k \ge 2$. Because every boundary vertex has degree one, each $v_i$ must be an internal vertex. By Proposition~\ref{prop:the plabic graph of a rational Dyck positroid}, each internal vertex is connected to exactly one boundary vertex, which implies that $\deg(v_i) \ge 3$ for each $i=1,\dots,k$. On the other hand, it immediately follows by Proposition~\ref{prop:the plabic graph of a rational Dyck positroid} that every vertex of $G$ has degree at most $3$. Thus, $\deg(v_i) = 3$ for each $i=1,\dots,k$. As $G$ is connected the set of internal vertices of $G$ is $\{v_1, \dots, v_k\}$, contradicting the fact that $G$ has internal vertices of degree $2$, for instance, the black internal vertex adjacent to the boundary vertex $1$.
	
	The fact that no graph in the move-equivalence class of the plabic graph described in Proposition~\ref{prop:the plabic graph of a rational Dyck positroid} has an internal cycle immediately implies that (R1) cannot be applied to any of such graphs. Hence the plabic graph described in Proposition~\ref{prop:the plabic graph of a rational Dyck positroid} must be reduced.
\end{proof}

\begin{theorem} \label{thm:main theorem on plabic graphs}
	A rank $d$ positroid on the ground set $[d+m]$ is a rational Dyck positroid if and only if it can be represented by one of the plabic graphs $G_\mathsf{d}$ described in Proposition~\ref{prop:the plabic graph of a rational Dyck positroid}.
\end{theorem}

\begin{proof}
	Let us prove first that for every rational Dyck path $\mathsf{d}$ of type $(m,d)$, the plabic graph $G_\mathsf{d}$ represents a rational Dyck positroid. To do this we will show that the decorated trip permutation $\pi_G$ of $G_\mathsf{d}$ is precisely the decorated permutation $\pi_\mathsf{d} = (1 \ j_2 \ \dots j_{d+m})$ of the rational Dyck positroid $P$ induced by $\mathsf{d}$. Note that, if we label the internal vertices of $G_{\mathsf{d}}$, which lie on $\mathsf{d}$, as in Proposition~\ref{prop:decoding decorated permutation}, then the endpoints of each edge of $G_{\mathsf{d}}$ incident to the boundary have the same label.
	
	First, we suppose that for $n < d+m$ the $n$-th step (going southwest) of $\mathsf{d}$, which is labeled by $j_n$, is vertical. If the $(n+1)$-th step of $\mathsf{d}$ is also vertical, then $\pi_G(j_n) = j_{n+1}$ as there is a path in $G_{\mathsf{d}}$ from $j_n$ to $j_{n+1}$ following the rules of the road, namely the path going from the boundary vertex $j_n$ west to the black internal vertex in the middle of the $n$-th step of $\mathsf{d}$, turning left to the black internal vertex in the middle of the $(n+1)$-th step of $\mathsf{d}$, and turning left to the boundary vertex $j_{n+1}$. On the other hand, if the $(n+1)$-th step of $\mathsf{d}$ is horizontal, then there is also a path from $j_n$ to $j_{n+1}$  following the rules of the road, namely the one going from the boundary vertex $j_n$ to the black internal vertex in the middle of the $n$-th step of $\mathsf{d}$, turning left to the white internal vertex in the middle of the $(n+1)$-th step of $\mathsf{d}$, and turning right to the boundary vertex $j_{n+1}$, yielding again $\pi_G(j_n) = j_{n+1}$. In a similar way we can argue that $\pi_G(j_n) = j_{n+1}$ when the $n$-th step of $\mathsf{d}$ is horizontal; the verification is left to the reader.
	
	Also notice that the path in $G_{\mathsf{d}}$ starting at the boundary vertex labeled by $d+1$ must travel in northeast direction through all the internal vertices until it reaches the boundary vertex labeled by $1$; this is because every time it visits a black (resp., white) internal vertex it must turn left (resp., right) and this forces the path to avoid the edges incident to the boundary (except the first one and last one). Hence $\pi_G(d+1) = 1$ and, therefore, we can conclude that $\pi_G$ is the decorated permutation of the rational Dyck positroid induced by $\mathsf{d}$
	
	We have proved that each of the plabic graphs of type $(d,d+m)$ described in Proposition~\ref{prop:the plabic graph of a rational Dyck positroid} represents a rational Dyck positroid of rank $d$ on the ground set $[d+m]$. On the other hand, given a positroid induced by a rational Dyck path $\mathsf{d}$ of type $(m,d)$, we can find its decorated permutation $\pi_\mathsf{d}$ by reading $\mathsf{d}$ in southwest direction as in Proposition~\ref{prop:decoding decorated permutation} and, by the argument just explained above, one finds that $\pi_G = \pi_\mathsf{d}$, where $\pi_G$ is the decorated trip permutation of the plabic graph $G_{\mathsf{d}}$ obtained from $\mathsf{d}$ by following the recipe in the statement above. The proof now follows.
\end{proof}

%
%
%

\begin{example} \label{ex:reducing RD plabic graph to the smallest bipartite}
	Let $P$ be the positroid induced by the rational Dyck path $\mathsf{d}$ of type $(8,5)$ shown in Figure~\ref{fig:rational Dyck path}. The following picture illustrates the plabic graph $G_{\mathsf{d}}$ corresponding to $P$ described in Proposition~\ref{prop:the plabic graph of a rational Dyck positroid} (on the left) and a minimal bipartite graph in the move-equivalence class of $G_{\mathsf{d}}$ (on the right). 
	\begin{figure}[h]
		\centering
		\includegraphics[width = 5.7cm]{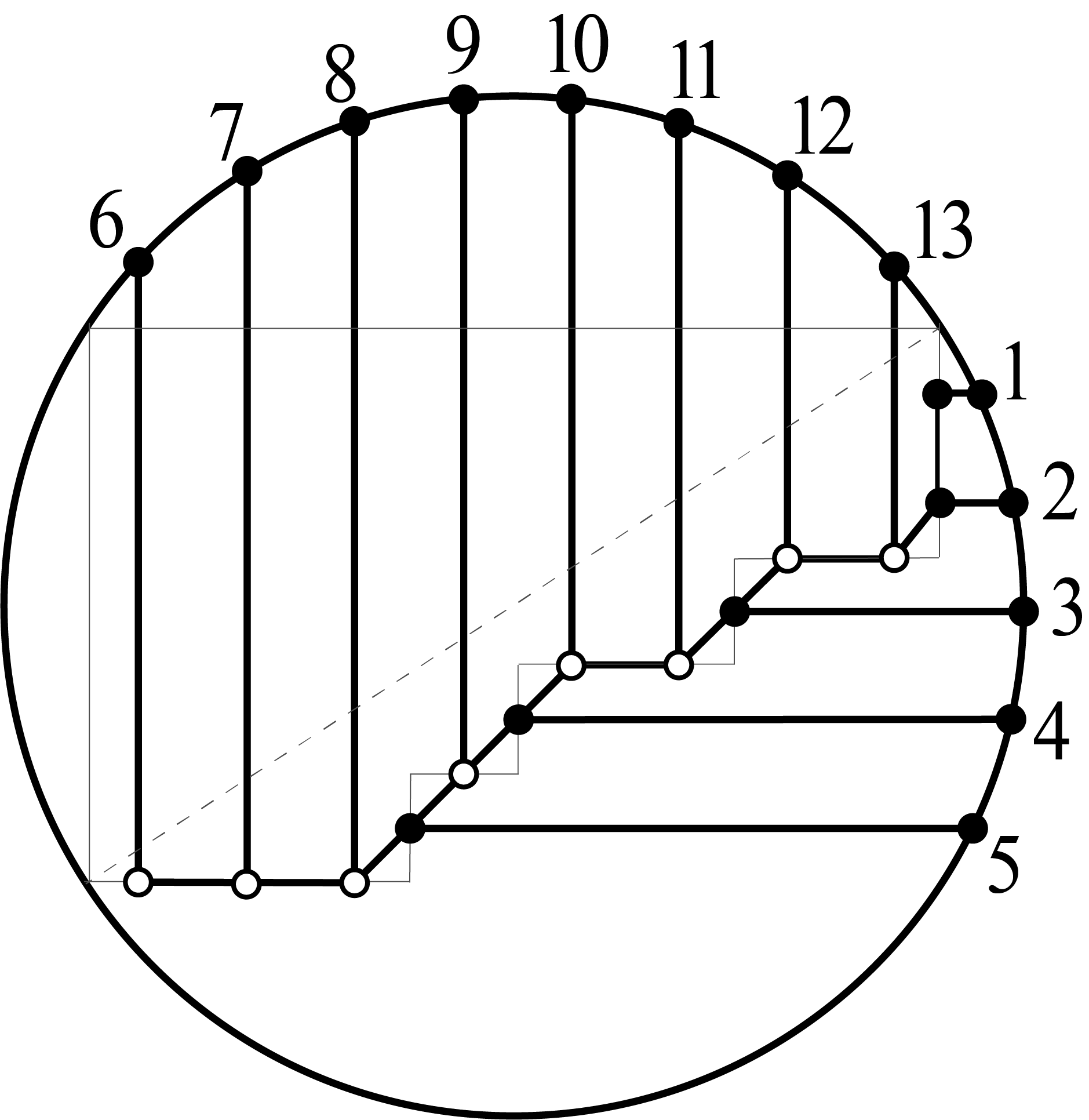} \hspace{1.6cm}
		\includegraphics[width = 5.7cm]{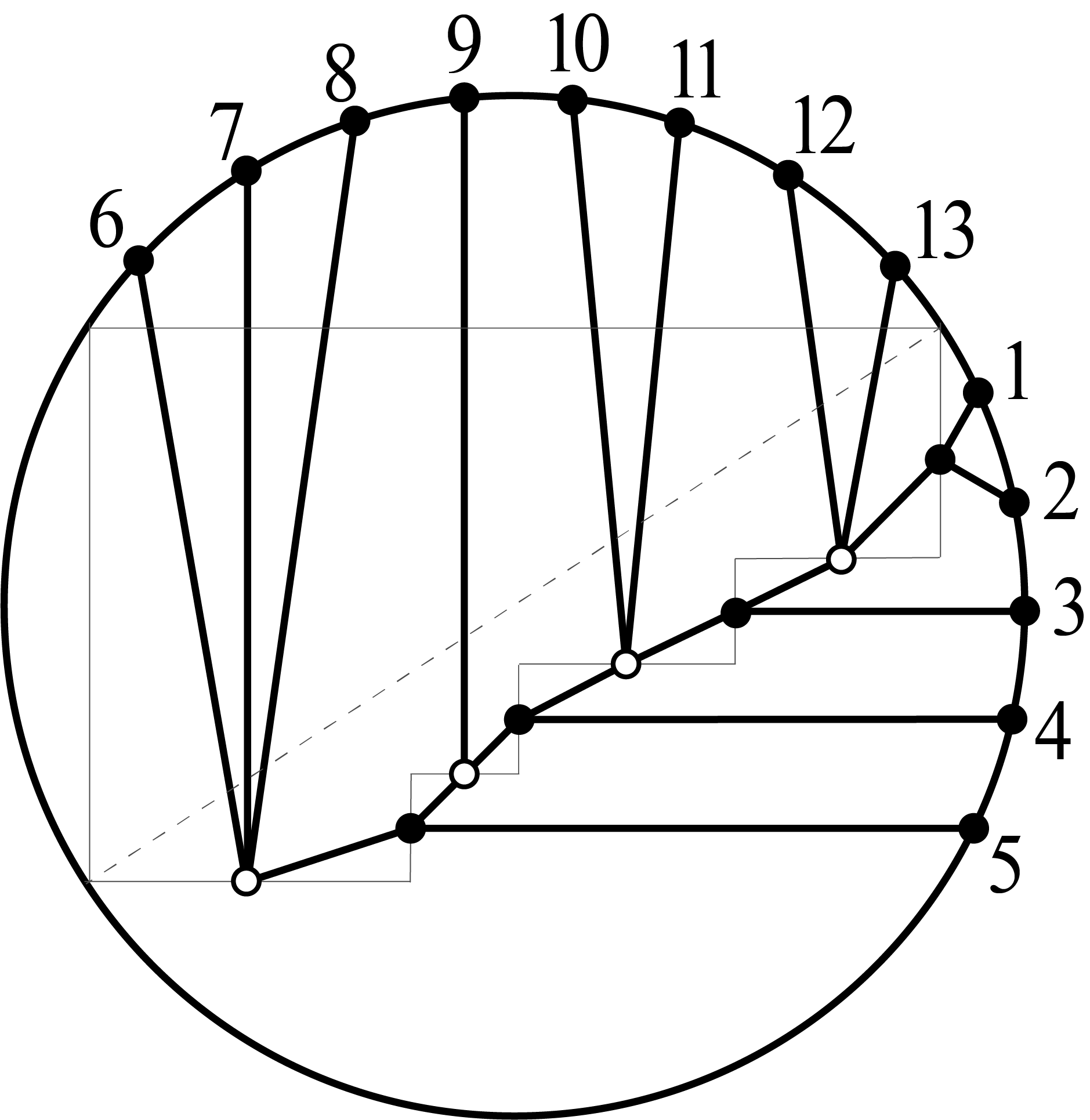}
		\caption{The plabic graph of a rank $5$ rational Dyck positroid on the ground set $[13]$ and a minimal bipartite plabic graph of its move-equivalence class.}
		\label{fig:plabic graph of a rational Dyck positroid and its reduced representative.}
	\end{figure}
\end{example}

\section{The Polytope of a Rational Dyck Positroid} \label{sec:positroid polytope}

In this section we characterize the matroid polytope of a rational Dyck positroid. We will do this by refining the description given in \cite{ARW16} of the matroid polytope of a general positroid by a set of inequalities.

The \emph{indicator vector} of a subset $B$ of $[n]$ is defined to be $e_B := \sum_{j \in B} e_j$, where $e_1, \dots, e_n$ are the standard basic vectors of $\rr^n$. Also, for a subset $S$ of $\rr^n$, we let $\text{conv} (S)$ denote the convex hull of $S$.

\begin{definition}
	The \emph{matroid polytope} $\Gamma_M$ of the matroid $M = ([n], \mathcal{B})$ is the convex hull of all indicator vectors of subsets in $\mathcal{B}$, namely
	\[
		\Gamma_M := \text{conv}\big(\{e_B \mid B \in \mathcal{B}\}\big).
	\]
	When $M$ happens to be a positroid, we call $\Gamma_M$ the \emph{positroid polytope} of $M$.
\end{definition}

Matroid polytopes have been extensively studied in the literature; see, for example, \cite{GGMS87}, \cite{ABD10} and references therein. In particular, the reader will find the next elegant characterization.

\begin{theorem} \cite[Theorem~4.1]{GGMS87}
	Let $\mathcal{B}$ be a collection of subsets of $[n]$, and let $\Gamma_{\mathcal{B}}$ denote $\emph{conv}\big(\{e_B \mid B \in \mathcal{B}\}\big) \subset \rr^n$. Then $\mathcal{B}$ is the collection of bases of a matroid if and only if every edge of $\Gamma_{\mathcal{B}}$ is a parallel translate of $e_i - e_j$ for some $i, j \in [n]$.
\end{theorem}

Descriptions of matroid polytopes by sets of inequalities have also been established. For instance, in \cite{dW76} Welsh describes a general matroid polytope $M = ([n], \mathcal{B})$ by using $O(2^n)$ inequalities. When the matroid $M$ happens to be a positroid, its positroid polytope can be described by using only $O(n^2)$ inequalities.

\begin{prop} \cite[Proposition~5.5]{ARW16} \label{prop:positroid polytope}
	Let $\mathcal{I} = (I_1, \dots, I_n)$ be a Grassmann necklace of type $(d,n)$, and let $M$ be its corresponding positroid. For any $j \in [n]$, suppose the elements of $I_j$ are $a_1^j \le_j  \dots \le_j a_d^j$. Then the positroid polytope $\Gamma_M$ can be described by the inequalities
	\begin{align}
		x_1 + x_2 + \dots + x_n &= d, \hspace{1cm}& \label{eq:positroid polytope 1} \\
		x_j &\ge 0 &\text{for each} \ j \in [n], \label{eq:positroid polytope 2}  \\
		x_j + x_{j+1} + \dots + x_{a_k^j - 1} &\le k-1 &\text{for each} \ j \in [n] \ \text{and} \ k \in [d], \label{eq:positroid polytope 3} 
	\end{align}
	where all the subindices are taken modulo $n$.
\end{prop}

Our next task consists in refining Proposition~\ref{prop:positroid polytope} for those positroids induced by rational Dyck paths; we will accomplish this by detecting redundant inequalities.

\begin{prop} \label{prop:rational Dyck positroid polytope}
	Let $P$ be a rational Dyck positroid represented by the real $d \times (d+m)$ matrix $A \in \phi_{d,m}(\mathcal{D}_{d,m})$, and let $I_A = \{p_1 < \dots < p_t\}$ be the set of principal indices of $A$. Then the positroid polytope $\Gamma_P$ is described by the inequalities
	\begin{align}
		x_1 + \dots + x_{d+m}              &= d,		& \label{eq:RDP polytope 1} \\
		x_i     									  &\ge 0	  &\text{for} \ i \in [d+m], \label{eq:RDP polytope 2} \\
		x_i 		     							  &\le 1        &\text{for} \ i \in [d], \label{eq:RDP polytope 3} \\
		x_{p_i} + \dots + x_{p_{i+1}-1}  &\le 1        &\text{for} \ i \in [t], \label{eq:RDP polytope 4} \\
		x_i + \dots + x_{p_{m(i)} - 1} 	  &\le (d-i)+ m(i) &\text{for} \ i \in [d], \label{eq:RDP polytope 5} \\
		x_{p_i} + \dots + x_{\omega_A(p_i)} & \le \omega_A(p_i)   &\text{ for} \ i \in [t] \! \setminus \! \{1\},
		\label{eq:RDP polytope 6}
	\end{align}
	where $m(i) = \max\{r \in [t] \mid \omega_A(p_r) \ge i \ \text{and} \ r < i\}$.
\end{prop}

\begin{proof}
	Let $\mathcal{I} = (I_1, \dots, I_{d+m})$ be the Grassmann necklace corresponding to $P$, and let $\Gamma$ be the polytope determined by \eqref{eq:RDP polytope 1}--\eqref{eq:RDP polytope 6}. Take $x = (x_1, \dots, x_{d+m})$ in $\Gamma_P$. Both \eqref{eq:RDP polytope 1} and \eqref{eq:RDP polytope 2} hold by Proposition~\ref{prop:positroid polytope}. Besides, taking $j \in [d]$ and $k = 2$ in \eqref{eq:positroid polytope 3}, we obtain the inequalities \eqref{eq:RDP polytope 3}, while taking $j \in I_A$ and $k=2$ we get the inequalities \eqref{eq:RDP polytope 4}. By Proposition~\ref{prop:grassmann necklace description}, for $i \in [d]$, the first $d-i+1$ entries of $I_i$ are $i, \dots, d$ and the next entries are some of the indices $p_1, \dots, p_t$. Therefore taking $j = i$ and $k= (d-j+1) + m(i)$ in the inequality \eqref{eq:positroid polytope 3}, one gets \eqref{eq:RDP polytope 5}. Again, by part (4) of Proposition~\ref{prop:grassmann necklace description}, the first $t-i+1$ entries of $I_{p_i}$ are $p_i, \dots, p_t$ and the next $\omega_A(p_i) + 1 - (t-i+1)$ entries of $I_{p_i}$ are the indices in $[\omega_A(p_i)] \setminus \{\omega_A(p_r) \mid i \le r \le t\}$. Hence \eqref{eq:RDP polytope 6} follows.
	
	Now we show that every element $x = (x_1, \dots, x_{d+m}) \in \Gamma$ must satisfy \eqref{eq:positroid polytope 1}--\eqref{eq:positroid polytope 3}. As \eqref{eq:positroid polytope 1} and \eqref{eq:positroid polytope 2} hold by the definition of $\Gamma$, it suffices to verify the inequality \eqref{eq:positroid polytope 3}. For $j \in [d+m]$ let $I_j = (a^j_1, \dots, a^j_d)$. We always have $a^j_1 = j$. Suppose first that $j \in [d]$. If $j < a^j_k \le d+1$, then \eqref{eq:positroid polytope 3} results from adding $k-1$ of the inequalities \eqref{eq:RDP polytope 3}. If $a^j_k = 1$ or $a^j_k > d+1$, then $a^j_{k-1} = p_{k'}$ for $k' \in [t]$. In this case, \eqref{eq:positroid polytope 3} results from adding \eqref{eq:RDP polytope 5} for $i=j$ and $k'-2$ inequalities \eqref{eq:RDP polytope 4}. Now suppose that $a^j_k = r$, where $1 < r < j$. Then we can obtain \eqref{eq:positroid polytope 3} by adding \eqref{eq:RDP polytope 5} for $i=j$, $k-3$ inequalities \eqref{eq:RDP polytope 4}, and \eqref{eq:RDP polytope 6} for the index $i$ such that $\omega_A(p_i) = a^j_k - 1$.
	
	Finally, suppose that $j \in \{d+1, \dots, d+m\}$. We can always assume that $a^j_1 \in I_A$ because $a^j_1 \in \{d+1,\dots, d+m\} \! \setminus \! I_A$ gives redundant inequalities. Let $a^j_1 = p_s$ for some $s \in [t]$. If $a^j_k \in \{p_s + 1, \dots, d+m\} \cup \{1\}$, then by Proposition~\ref{prop:grassmann necklace description}, it follows that $a^j_k \in I_A \cup \{1\}$; in this case \eqref{eq:positroid polytope 3} is the addition of $k-1$ inequalities \eqref{eq:RDP polytope 4}. It only remains to consider $a^j_k \in \{2,\dots,d\}$. Suppose first that $a^j_k \le \omega_A(p_s)$. If $a^j_k \le \omega_A(p_n)$ for each $n \in [t]$, then \eqref{eq:positroid polytope 3} is the addition of $t-s+1$ inequalities \eqref{eq:RDP polytope 4} and $(k-1) - (t-s+1)$ inequalities \eqref{eq:RDP polytope 3}. Otherwise, there exists a smallest index $r$ in $[t]$ such that $a^j_k > \omega_A(p_r)$. Taking $i = p_r$, we observe that \eqref{eq:positroid polytope 3} is obtained from adding $i-s$ inequalities \eqref{eq:RDP polytope 4}, the inequality \eqref{eq:RDP polytope 6}, and enough inequalities \eqref{eq:RDP polytope 3}. Lastly, suppose that $a^j_k > \omega_A(p_s)$. In this case it is not hard to see that \eqref{eq:positroid polytope 3} is implied by the addition of \eqref{eq:RDP polytope 6} for $i=s$ and enough inequalities \eqref{eq:RDP polytope 3}.
\end{proof}

\begin{remark}
	Although the description of the rational Dyck positroid given in Proposition~\ref{prop:rational Dyck positroid polytope} is not as simple as the one presented in Proposition~\ref{prop:positroid polytope}, the reader should observe that the number of inequalities in Proposition~\ref{prop:rational Dyck positroid polytope} is $O(d+m)$ while the number of inequalities in Proposition~\ref{prop:positroid polytope} is $O(d^2 + dm)$.
\end{remark}

\section{Acknowledgements}

	I want to thank Bruno Benedetti and Maria Gillespie for suggesting me to look at the connection between rational Dyck paths and positroids, and my advisor, Lauren Williams, for many helpful conversations and for encouraging me to work on this project.

\end{document}